\documentclass[11pt,reqno]{amsart}
\usepackage{amssymb}
\usepackage{amsfonts}
\usepackage{amsthm}
\usepackage{amsmath}
\usepackage[pagebackref,hypertex]{hyperref} 
\usepackage{backref}

\DeclareMathOperator*{\dist}{dist}

\DeclareMathOperator*{\supp}{supp}

\newcounter{smalllist}

\numberwithin{equation}{section}
\allowdisplaybreaks

\newtheorem{thm}{Theorem}
\newtheorem{rem}[thm]{Remark}
\newtheorem{lemma}[thm]{Lemma}

\newtheorem{claim}[thm]{Claim}
\newtheorem{estimate}[thm]{Estimate}
\newtheorem{defn}[thm]{Definition}

\newcommand{\bbr}{\mathbb R}
\newcommand{\bbz}{\mathbb Z}

\newcommand{\cala}{\mathcal A}
\newcommand{\calb}{\mathcal B}
\newcommand{\cald}{\mathcal D}

\newcommand{\cali}{\mathcal I}
\newcommand{\calj}{\mathcal J}
\newcommand{\calp}{\mathcal P}

\newcommand{\calr}{\mathcal R}
\newcommand{\calt}{\mathcal T}
\newcommand{\cals}{\mathcal S}
\newcommand{\calu}{\mathcal U}
\newcommand{\calc}{\mathcal C}

\newcommand{\one}{\mathbf{1}}

\newcommand{\boldc}{\mathbf{C}}

\newcommand{\dense}{\mathbf {dense} }
\newcommand{\size}{\mathbf {size} }
\newcommand{\udense}{\overline{\dense}}
\newcommand{\topp}{\mathbf {top} }

\newcommand{\beq}{\begin{equation}}
\newcommand{\eeq}{\end{equation}}
\newcommand{\beqa}{\begin{eqnarray*}}
\newcommand{\eeqa}{\end{eqnarray*}}
\newcommand{\beqan}{\begin{eqnarray}}
\newcommand{\eeqan}{\end{eqnarray}}

\begin{document}

\title[Single annulus $L^p$ estimates]{Single annulus $L^p$ estimates for Hilbert transforms along vector fields}
\author{Michael Bateman}
\address{Michael Bateman, UCLA}
\email{bateman@math.ucla.edu}

\begin{abstract}
We prove $L^p$, $p\in (1,\infty)$ estimates on the Hilbert transform along a one variable vector field acting on functions with frequency support in an annulus.  Estimates when $p>2$ were proved by Lacey and Li in \cite{LL1}.  This paper also contains key technical ingredients for a companion paper \cite{BT} with Christoph Thiele in which $L^p$ estimates are established for the full Hilbert transform.  
\end{abstract}

\maketitle

\tableofcontents

\section*{Acknowledgments} 
The author thanks Ciprian Demeter and Christoph Thiele for helpful discussions, and especially Christoph Thiele for making many comments on an early version of this paper.  The author also thanks Francesco Di Plinio for pointing out an important typo. The author was partially supported by NSF grant DMS-0902490.


\section{Introduction}

Let $v$ be a nonvanishing vector field that depends one variable, i.e., $v\colon \bbr ^2 \rightarrow \bbr ^2 \setminus \{0\}$ and $v(x_1, x_2)=v(x_1)$.  In this paper we prove $L^p$ estimates on the Hilbert transform along $v$ precomposed with frequency restriction to an almost-annular region.  More specifically, define
\beqa
H_v f(x) = p.v. \int {{f(x- t v(x) ) } \over t} dt.
\eeqa
Because of the structure of the Hilbert kernel, the magnitude of $v$ is irrelevant, provided it is nonzero.  For this reason we may assume that $v(x_1, x_2) = (1, u(x_1))$.  
We will further assume that the slope of $v$ is bounded by $1$.  This will be helpful for some technical reasons in this paper, but our main interest is in the action of 
$H_v$ on arbitrary functions (i.e., those not necessarily having frequency support in an annulus); in this more general case, the operator is invariant under dilations in the vertical variable.  See \cite{BT} for more on the symmetries of this problem.  This invariance allows us to assume, in that case, that the slope of $v$ is bounded by $1$.  (This is mostly a technical convenience, that allows us to think of rectangles and parallelograms as being the same kind of objects.)  Since this general problem is the primary motivation for this paper, we adopt the restriction on the slope here as well.  The general problem is addressed in a companion paper with Christoph Thiele \cite{BT}.  This paper is logically prior to the other, and is therefore self-contained.  Fix $w\geq 0$, and define $\tau$ to be the trapezoid with corners $(-{1\over w}, {1\over w}),$ $({1\over w}, {1\over w}),$ 
$(-{2\over w}, {2\over w}),$ and $({2\over w}, {2\over w})$.  Also define
\beqa
\widehat{\Pi _{\tau}  f }(\xi) = \one _{\tau} (\xi) \hat{f} (\xi).
\eeqa
Here we prove the following
\begin{thm}\label{main}
Let $v$ be a vector field depending on one variable with slope bounded by $1$.
Let $p\in (1,\infty)$.  Then
\beqa
||(H_v \circ \Pi _{\tau} ) f || _p \lesssim || \Pi _{\tau} f || _p .
\eeqa
\end{thm}
We remark that the estimate in this theorem is independent of the parameter $w$ in the definition of $\tau$, which comes as no surprise given the dilation invariance of the problem.  Further, the restriction to a trapezoid specifically is nothing to take seriously.  Using the assumption on the slope of the vector field we can already assume $\supp \hat{f}$ lies in a two-ended cone near the vertical axis, because $H_v$ acts trivially on functions with support outside this cone.  More precisely, if $\hat{f}$ is support in a cone close to the horizontal axis, then we have with the constant vector field $(1,0)$:
\beq 
\label{honezero}
H_vf(x,y)=H_{(1,0)}f(x,y)\ \ ,
\eeq 
because $H_{(1,0)}$ is a multiplier corresponding to right and left half-planes.  But $H_{(1,0)}$ is trivially bounded, justifying our claim.  Finally, a trapezoid is the restriction of the cone to a horizontal frequency band.  We could have equally well stated the theorem for functions with support in the full band, and reduced it to the trapezoidal case.  Alternatively, we could have worked with an annular region, or an annular region intersected with a cone.  Our methods work equally well in these cases.  We chose the horizontal band (rather than an annulus) because of the special structure of one-variable vector fields, but for other vector fields an annular region may be more appropriate.

Perhaps the biggest contribution of this paper (aside from its applicability to \cite{BT}) is a more streamlined and mechanized collection of two-dimensional time-frequency tools.  Building heavily on important earlier work of Lacey-Li (see \cite{LL1} and \cite{LL2}), we clarify the relationship between the density-related maximal operators (see Lemma \ref{maximallemma}) and the more classical time-frequency tools.  Specifically, a key sublemma in \cite{B1}, combined with this more efficient understanding, allows us to obtain the full range of exponents $p \in (1, \infty)$ here.  Further, although the results are stated only for one-variable vector fields, it is clear how to combine a maximal theorem for a different vector field with the methods of this paper.  We should remark that time-frequency analysis in two-dimensions is rather less-well-developed than in one-dimension, with work of Lacey-Li being the only natural precursor to this paper.  We therefore strove to make the paper self-contained and to include proofs of a number of lemmas that are standard in one-dimension, but whose proofs in the two-dimensional situation do not seem to appear in the literature.  

\subsection{Related work}
Study of such problems is motivated by the obvious connection to the problem of estimating the Hilbert transform on functions that have not been Fourier-localized.  Stein, for example, conjectured that if $v$ is Lipschitz, then $H_v$ (or rather, a truncated version of it) is a bounded operator on $L^2$.  We note that when $v$ depends on only one variable, the $L^2$ boundedness of $H_v$ is a rather immediate consequence of Carleson's theorem, as shown in \cite{LL2}.  Stein's conjecture is the singular integral variant of Zygmund's well-known conjecture on the differentiation of Lipschitz vector fields.  For a fuller history, see \cite{LL2}.  More recently, Thiele and the author proved a range of $L^p$ estimates on the full Hilbert transform along a one variable vector field, using some key lemmas from the present paper.  It is known that the operator $H_v$ is related to the return-times theorem from ergodic theory; see \cite{BT} for more on this connection.

We remark that the operator $\boldc$ is quite similar to Carleson's operator (i.e., the maximal Fourier partial sum operator).  The argument in \cite{LL1} is also quite similar to the Lacey-Thiele proof of Carleson's theorem (see \cite{LT}).  The argument here draws on ingredients from \cite{LL1}, but obtaining $L^p$ estimates for $p<2$ in this situation requires more effort, partly because the relevant maximal operators are more complicated, but also because making use of the maximal theory is more complicated.  In the $1$-D situation, exceptional sets are unions of intervals; nothing so simple is the case here.

Theorem \ref{main} was proved for arbitrary vector fields when $p>2$ by Lacey and Li in \cite{LL1}.  (In fact, they proved a weak $L^2$ result.)  The same authors, in \cite{LL2}, introduced a method for obtaining $L^p$, $p<2$, estimates on $H_v \circ \Pi _{\tau} $ when a certain maximal theorem is available for the vector field $v$ in question.  (The story is a bit technical:  they proved a theorem contingent on the existence of this certain maximal theorem in the case of truncated Hilbert kernels.  However the method had little to do with the truncation of the kernel, allowing us to extend it here.)  The author proved such an $L^p$ maximal theorem when $v$ depends on one variable in \cite{B1}.  Given this result, it is not surprising that the method from \cite{LL2} yields a result for some $p<2$, but the value of $p$ obtained from the method in \cite{LL2} seemed far from sharp.  (At the very least, the method seemed nonsharp.  Of course, this was not important for the authors there.)  It was clear, for example that new ideas would be required to even reach $p$ close to ${3\over 2}$.  The author recently improved the estimates in this maximal theorem to (essentially) best possible in \cite{B2}.  Because of this, the author decided to investigate the precise range of $p$ for which Theorem \ref{main} holds.


\subsection{New ideas}
The novelties in this paper that allow us to obtain the full range of $p$ claimed in Theorem \ref{main} are
 a simplification of the approach in \cite{LL2}, and
   a more efficient appeal to the maximal theorems.

We elaborate a bit more on these points for readers already familiar with the argument in \cite{LL2}.

Regarding the first point:  In \cite{LL2}, tiles are sorted into trees via standard density and orthogonality (size) lemmas.  An important additional observation made in \cite{LL2} is that if $\calt$ is a collection of trees such that for each $T \in \calt $ the ``size" of $T$ is about $\sigma$ and the ``density" of the top of $T$ is about $\delta$, then we can control $\sum _{T\in \calt} |\topp (T)| $ by using an appropriate maximal theorem.  Their argument, however, requires an additional twist to handle trees with large size whose tiles have density $\sim \delta$, but whose tops have density much less than $\delta$.  Here we use an organization of the tiles that admits a more straightforward argument.  This organization is carried out in Section \ref{maximalsection}, which contains more discussion as well.  

Regarding the second point:  A rather simple observation allows us to appeal to a key ingredient in the proof of the maximal theorem, rather than the theorem itself.  This strengthens estimates on $\sum _{T\in \calt} |\topp (T)| $ for trees as mentioned in the last paragraph.  This observation allows us to obtain the full range of $p$.  This observation uses the proof of \cite{B1}, and hence does not even take advantage of the sharp $L^p$ estimates on the maximal operator obtained in \cite{B2}.  See Lemma \ref{maximallemma}.

\subsection{Organization of paper}
Readers familiar with time-frequency analysis, having a bit of faith, and wanting an executive summary should follow this outline:  Skip to the definition of the model operator in Section \ref{modeloperatorsection}.  Then (possibly after skimming Section \ref{defs} to review essentially standard definitions,) read Sections \ref{organizationsection}, \ref{lemmas}, and \ref{maximalsection}.  Those wanting to check the numerology should also read Section \ref{optimization}.  A comprehensive outline is below.

In Section \ref{modelreduction}, we reduce the theorem to an analogous one for a model operator.

In Section \ref{defs}, we present some key definitions needed for the organization of our set of tiles.  (Recall that the operators in question are model sums over tiles.)  

In Section \ref{organizationsection}, we make the main decomposition of the collection of tiles and state several key estimates that follow from the decomposition.

In Section \ref{lemmas}, we state the main lemmas needed to prove the estimates stated in Section \ref{organizationsection}.

In Section \ref{optimization}, we balance these various estimates to prove the main theorem.  There is no serious content here.

In Section \ref{densitysection},  we prove the density lemma, which estimates $\sum _{T\in \calt} |\topp (T)| $ for certain collections $\calt$ by using elementary covering ideas.

In Section \ref{maximalsection}, we prove the maximal estimate, which controls $\sum _{T\in \calt} |\topp (T)| $ for certain collections $\calt$ by using more sophisticated techniques in combination with $L^p$ and $BMO$-type estimates on a square function related to the ``projection" operator associated to trees.

In Section \ref{intersectionsection}, we compare the $\size$ of a tree to its intersection with the function in the definition of $\size$.

In Section \ref{treesection}, we prove the tree lemma, which controls the contribution to the model sum from one tree.  The proof mirrors that of the (more) classical one-dimensional tree lemma, with a small bit of extra work required to handle two-dimensional tail terms.

In Section \ref{sizesection}, we prove the size lemma, which estimates $\sum _{T\in \calt} |\topp (T)| $ for certain collections $\calt$ by using orthogonality.

In Section \ref{besselsection}, we prove a refined Bessel inequality that allows us to control tail terms in the size and tree lemmas, as well as in the proof of localized $L^p$ estimates for the square function mentioned above.

In Section \ref{sqfcnsection}, we prove localized (to the top of a tree) $L^p$ estimates for a square function associated to a tree.  Once again, we follow a relatively standard argument and appeal to the refined Bessel inequality to handle some two-dimensional technicalities.

In Section \ref{bmosection}, we prove that higher $L^p$ norms of the square function are controlled by lower $L^p$ ones by using standard $BMO$ techniques.

In the appendix, we recall the proof in \cite{LL1} of the $L^p$, $p>2$ case of our main theorem.


\section{Reductions } \label{modelreduction}

In this section we reduce the $L^p$ estimates in Theorem \ref{main} to restricted weak-type estimates on a model operator.  The model operator should look familiar to readers familiar with developments in time-frequency analysis from the last ten to fifteen years:  it is a sum over ``tiles" of wave packets.  The model operator arises from decomposing 
\begin{enumerate}
\item the Hilbert kernel  ${1\over t}$ into (smoothly cutoff) dyadic intervals on the frequency side; for technical reasons we make these annuli rather thin, resulting in two summation indices for the Hilbert kernel.  In fact, we actually decompose the projection operator onto positive frequencies, and write the Hilbert transform as a linear combination of this operator and the identity operator. 
\item  given any integer $l\geq 0$, $\hat{f}$ on $\tau$ into $\sim 2^{l}$ pieces; again, the ``$\sim$" here comes from another summation introduced to provide strict orthogonality between the various pieces.
\end{enumerate}


\subsection{Discretizing the kernel}

In this section we decompose the operator $H \circ \Pi _{\tau} $ into a sum of model operators.

We begin by selecting a Schwartz function $\psi _0 ^{(0)}$ such that $\psi_0 ^{(0)} $ is supported on $[{{98}\over {100}}, {{102}\over{100}}]$ and equal to $1$ on $[{{99}\over{100}}, {{101}\over{100}}]$ .  Let $\psi_l ^{(0)} (t)=  \psi ^{(0)} (2^{-l} t) $.  Now  define  $\psi ^{(0)} = \sum _{l\in \bbz} \psi ^{(0)} _l$.  By appropriately defining $\psi _0 ^{(i)} $ with similarly sized support, and defining $\psi _l ^{(i)} (t)=  \psi _0 ^{(i)} (2^{-l} t)$, we can construct a partition of unity for $\bbr ^+$; i.e. 
\beqa
\one_{(0, \infty)} = \sum _{i=0} ^{99} \psi ^{(i)}  .
\eeqa
This gives us the Hilbert kernel as a linear combination of $100$ model kernels and the identity.  
More precisely, let 
\beqa
H_l ^{(i)} g (x,y) = \int \check{\psi} _l ^{(i)} (t) g(x-t, y - t u (x)) dt.
\eeqa
Then writing $I$ for the identity operator, 
\beqa
c_1H \circ \Pi _{\tau} f (x,y) + I \circ \Pi _{\tau} f (x,y)&=&   c_2 \sum _{l\in \bbz} \sum _{i=0} ^{99} H_l ^{(i)} \circ \Pi _{\tau} f (x,y) .\\
\eeqa 
By the triangle inequality, we have 
\beqa
||H \circ \Pi _{\tau} f ||_p \lesssim ||I\circ \Pi _{\tau}f||_p +  \sum _{i=0} ^{99} ||H ^{(i)}  \circ \Pi _{\tau}f || _p ,
\eeqa
where $H^{(i)} = \sum _l H^{(i)} _l$.  We note that $H_l \circ \Pi _{\tau} f =0$ for $l\leq \log {1\over w} + c$ because of the Fourier support of the kernel of the operator $H_l$.

\subsection{Discretizing the function} \label{discfcn}

We next focus on discretizing the function $f$.  For $l\geq 0$, we write $\cald _l$ to denote the collection of dyadic intervals of length $2^{-l}$ contained in $[-2,2]$.
Fix a smooth positive function $\beta \colon \bbr \rightarrow \bbr$ such that $\beta (x) =1$ for $x\in [-1,1]$ and such that $\beta (x)=0$ when 
$|x|\geq 2$.  Also assume that $\sqrt{\beta}$ is a smooth function.  This point will become relevant for the definition of $\varphi$ immediately before Lemma \ref{fouriersummation}.  Now fix an integer $c$ (whose exact value is unimportant) and for each $\omega \in \cald_l$, define 
\beqa
\beta _{\omega} (x) = \beta (2^{l+c} (x - c_{\omega _1}) ),
\eeqa
where $\omega _1$ is the right half of $\omega $, and $c_{\omega_1}$ is the center of $\omega _1$.  Define 
\beqa
\beta _l (x) = \sum _{\omega \in \cald _l} \beta _{\omega } (x).
\eeqa
Note that 
\beqa
\beta_l (x+ 2^{-l}) = \beta _l (x)
\eeqa
for $x\in [-2,2-2^{-l}]$.  Now define 
\beqa
\gamma _l (x) = {1\over 2} \int_{-1} ^{1} \beta _l (x+ t) dt.
\eeqa
Because of the local periodicity mentioned above, we have that $\gamma _l (x) $ is constant for $x\in [-1,1]$; say $\gamma _{l} (x) = \delta $, where $\delta $ is a constant independent of $l$.  Hence
\beqa
{1\over \delta} \gamma _l (x) \one_{[-1,1]} (x) = \one_{[-1,1]} (x).
\eeqa
Define yet another multiplier $\tilde{\beta} \colon \bbr \rightarrow \bbr$ with support in $[{1\over 2}, {5\over 2}]$, and 
$\tilde{\beta} (x)= 1 $ for $x\in [1,2]$.  Just as $\gamma _l$ is an average over translates of $\beta _l$, so each $H^{(i)}$ is an average of model operators.  We define the corresponding multipliers on $\bbr ^2$:
\beqa
\widehat{m_{\omega}} (\xi, \eta) &=& \tilde{\beta} (\eta) \beta _{\omega} ({\xi \over \eta}) \\
\widehat{m_{l,t}} (\xi, \eta) &=& \tilde{\beta} (\eta) \beta _l (t+ {\xi \over \eta}) \\
\widehat{\frak{m}_l} (\xi, \eta) &=& \tilde{\beta}(\eta) \gamma _l ({\xi \over \eta}).
\eeqa
We know that for each $l$
\beqa
\frak{m}_l (\xi, \eta) \one _{\tau}(\xi, \eta) = \one _{\tau} (\xi, \eta)
\eeqa
for $(\xi, \eta) \in \tau$.  Note that for each $i$, 
\beqa
||H ^{(i)} (\Pi _{\tau}  \circ f) ||_p &=& ||\sum _l ( H_l ^{(i)} \circ \Pi _{\tau} ) ({1\over \delta} \frak{m}_l \ast f) ||_p \\
&=& ||{1\over 2 } \int_{-1} ^1 \sum _l ( H_l ^{(i)} \circ \Pi _{\tau} ) ( {1\over \delta} m_{l,t} \ast f) dt ||_p \\ 
&\leq &  {1\over 2 } \int_{-1} ^1   ||  \sum _l ( H_l ^{(i)} \circ \Pi _{\tau} )( {1\over \delta} m_{l,t} \ast f) || _p dt,
\eeqa
so it is enough to consider the discretized projections $m_{l,t}$.  In what follows, we will assume, without loss of generality, that 
$t=0=i$ and omit the dependence on $t$ and $i$.

\subsection{Constructing the tiles} \label{consttile}
For each $\omega \in \cald$ with $l\geq 0$, let $\calu _{\omega}$ be a partition of $\bbr ^2$ by parallelograms of width $w$ and 
length ${w\over |\omega|}$ whose long side has slope $\theta$, where $\tan \theta = c(\omega)$ and where $c(\omega)$ is the center of the 
interval $\omega$, and whose projection onto the $x$-axis is a dyadic interval.  We remark that $l<0$ need not be considered.  (See the remark immediately prior to Section \ref{discfcn}.  Note that the index $l$ plays a slightly different role there.)  Briefly, the parts of the Hilbert kernel whose frequency support is outside the interval $[-{1\over w} , {1\over w}] \subseteq \bbr$ ((i.e., $\psi _l$ for $l<\log {1\over w}$) have no interaction with our function $f$ whose frequency support is contained in the annulus of radius ${1\over w}$. Finally, let $\calu = \bigcup _{\omega \in \cald} \calu _{\omega}$.  If $s\in \calu _{\omega}$, we will write 
$\omega _s : = \omega $.

An element of $\calu $ is called a ``tile".  The following lemma, stated in essentially this form in \cite{LL1}, allows us to further discretize our operator into a sum over tiles. 
Let $R_{\omega}$ denote an element of $\calu _{\omega}$ containing the origin.  
Suppose $\varphi _{\omega}$ is such that $|\widehat{\varphi _{\omega} }|^2 = \widehat{m_{\omega}}$.  Note that $\varphi _{\omega}$ is smooth, by our assumption on the function $\beta$ mentioned above.  Further, each region 
\beqa
\{ (\xi, \eta) \colon {\xi \over \eta} \in \omega , \hspace{.2cm} \eta \in [1,2] \}
\eeqa
 can be obtained by a linear transformation of the trapezoid with corners $(-1,1), (1,1), (-2,2), (2,2)$, which ensures that the functions $\varphi _{\omega}$, with $\omega \in \cald := \cup _{l \geq 0} \cald _l$, satisfy uniform decay conditions.  To see this, consider the transformations
\beqa
A= \begin{pmatrix} M & 0 \\ 0 & M \end{pmatrix},
\eeqa
\beqa
B= \begin{pmatrix} \epsilon & 0 \\ 0 & 1 \end{pmatrix},
\eeqa
and
\beqa
C= \begin{pmatrix} 1 & \lambda \\ 0 & 1 \end{pmatrix}.
\eeqa
A composition of these three takes the trapezoid bounded by 
$(-1,1)$, $(1,1)$, $(-2,2)$, $(2,2)$ to the trapezoid bounded by 
$(M(\epsilon + \lambda), M)$, $(M(-\epsilon + \lambda), M)$, $(2M(\epsilon + \lambda), 2M)$, $(2M(-\epsilon + \lambda), 2M)$, which is precisely the area of support for $\varphi _{\omega}$ when $M$, $\epsilon$, and $\lambda$ are chosen appropriately.  
Define
\beqa
\varphi _s (p)&=&  \sqrt{|s|} \varphi _{  \omega } (p-c(s)).
\eeqa
Note that the functions $m_{\omega}$ are $L^{1}$ normalized, so the functions $\varphi _{s}$ are $L^2$ normalized.
\begin{lemma} \label{fouriersummation}
Using notation above, we have
\beqa 
f \ast   m_{\omega}  (x) =\lim _{N\rightarrow \infty} {1\over 4N^2} \int _{[-N, N]^2}
\sum_{s\in \calu _{\omega} }  \langle f , \varphi _s (p + \cdot ) \rangle \varphi _s (p + x) dp.
\eeqa
\end{lemma}

\begin{proof}

We compute directly:  
\beqa
f \ast  m_{\omega}  (x) 
&=& \int _{z\in\bbr ^2} f(z)\int_{p\in\bbr ^2} \overline{\varphi_{\omega} (p)} \varphi_{\omega} (p +x-z) dp dz \\
&=& \int _{z\in\bbr ^2} f(z) \sum_{s\in \calu _{\omega}} \int _{p\in s} \overline{\varphi_{\omega} (p+z)} \varphi_{\omega} (p +x) dp dz \\
&=&\sum_{s\in \calu _{\omega}} \int _{p\in s}    \int _{z\in\bbr ^2} f(z) \overline{\varphi_{\omega} (p+z)} dz \varphi_{\omega} (p +x) dp \\
&=&\sum_{s\in \calu _{\omega}} \int _{p\in s}  \langle f , \varphi_{\omega} (p + \cdot ) \rangle \varphi_{\omega} (p + x) dp \\
&=&\sum_{s\in \calu _{\omega}} {1\over |R_{\omega}|} \int _{p\in R_{\omega} } \langle f , \varphi _s (p + \cdot ) \rangle \varphi _s (p + x) dp  \\
&=& \lim _{N\rightarrow \infty} {1\over 4N^2} \int _{[-N, N]^2}
\sum_{s\in \calu _{\omega} }  \langle f , \varphi _s (p + \cdot ) \rangle \varphi _s (p + x) dp.
\eeqa
To see the last equality, note that the integrand is periodic in $p$, and the error (which arises from the fact that $[-N, N]^2$ will not exactly agree with the boundaries of the tiles $s$) goes to zero as $N\rightarrow \infty$.
\end{proof}
This lemma allows us to conclude (using the dominated convergence theorem) that 
\beqa
H_l (f \ast  m_{\omega} ) (x) = \lim _{N\rightarrow \infty} {1\over 4N^2} \int _{[-N, N]^2} 
	H_l \left( \sum_{s\in \calu _{\omega} }  \langle f , \varphi _s (p + \cdot ) \rangle \varphi _s (p + x) \right) dp.
\eeqa
This allows us to restrict attention to the model operator that we define shortly.  
Define
\beqa
\psi _s = \psi _{\log (length(s))}
\eeqa
and 
\beqa
\phi _s (x_1, x_2) = \int \check{ \psi} _s (t)  \varphi _s (x_1 - t , x_2 - t v(x)) dt.
\eeqa
We record the following fact for use in the proof of the tree lemma in Section \ref{treeproof}.
\begin{lemma}
We have $\phi _s (x) =0$ unless $v(x) \in \omega _{s,2}$. 
\end{lemma}
\begin{proof}
Use Plancherel's theorem and the Fourier supports of $\psi _s$ and $\varphi _s$. 
\end{proof} 
\subsection{The model operator} \label{modeloperatorsection}
We can finally define our model operator:
\beqa
\boldc  f &=& \sum _{s\in \calu } \langle f, \varphi _s \rangle  \phi _s . 
\eeqa
For readers following the executive summary:  $\varphi _s$ is a standard wave packet adapted to the tile $s$, and $\phi _s$ is the appropriate scale of the Hilbert transform acting on $\varphi _s$.  A good mental shortcut is to imagine 
$\phi_s (x) = \varphi _s (x) \one _{\omega _{s,2}} (u(x))$, an expression quite similar to one appearing in the Lacey-Thiele proof of Carleson's theorem.
By Lemma \ref{fouriersummation}, each operator $H^{(i)}$ is an average of models of the form $\boldc $.  Hence it is enough to prove the following theorem.  
\begin{thm} \label{modelthm}
With $\boldc $ defined immediately above, and $p\in (1 ,\infty )$, we have 
\beqan \label{modelmain}
||\boldc f || _ p \lesssim ||f||_p .
\eeqan
\end{thm}


By appealing to restricted weak-type interpolation, it suffices to prove
\beqa
|\langle \boldc  \one _F, \one _E \rangle | \lesssim |E|^{1-{1\over p}} |F|^{1\over{p}}
\eeqa
for arbitrary  $E, F\subseteq \bbr ^2 $ and $p \in (1, \infty)$.  Of course by the triangle inequality it suffices to prove the following inequality:
\beqan  \label{model}
\sum _{s\in \cals} | \langle \one _F, \varphi _s \rangle
			\langle \one _E , \phi _s \rangle | \lesssim
		|E|^{1-{1\over p}} |F|^{1\over{p}}
\eeqan
for any $p\in (1, \infty) $, any $E, F\subseteq \bbr ^2 $, and any finite $\cals \subseteq \calu$.  This is our task for the rest of the paper.  Lacey and Li have already proved this estimate for arbitrary vector fields when $p\geq 2$.  We discuss this proof in the appendix.  Note that for $p\leq 2$, we have
\beqa
|E|^{1-{1\over p}} |F|^{1\over p} = |E|^{1\over 2} |F|^{1\over 2}
	\left(  {{|F|} \over {|E|}}  \right) ^{ {1\over p} - {1\over 2} }
	\gtrsim |E|^{1\over 2} |F|^{1\over 2}
\eeqa
whenever $|F| \gtrsim |E|$ because $ {1\over p} - {1\over 2} >0$.  Hence our estimate is already proved when
$|F| \gtrsim |E|$, so we restrict attention to the case $|F| \leq c |E|$ for some small constant $c$.


\section{Key definitions}\label{defs}

\begin{defn}
Given a parallelogram $R$, we write $CR$ to denote the parallelogram with the same center as $R$ but dilated by a factor of $C$.
\end{defn}
\begin{defn}
Given two parallelograms $R_1$ and $R_2$ in $\calu$, we will write $R_1 \leq R_2$ whenever $R_1 \subseteq CR_2 $ and
$\omega_{R_2} \subseteq \omega _{R_1} $.  
\end{defn}
Recall that $\omega_{R}$ is defined in Section \ref{consttile}.
The exact value of $C$ in the last definitions is not important: $10$ is enough.  We need that if $R_1 \cap R_2 \neq \emptyset$ and $\omega _{R_1} \subseteq \omega _{R_2} $, then $R_2 \leq R_1$.  


\begin{defn}\label{treedef} A \rm tree \it is a collection $T$ of parallelograms with a top parallelogram, denoted $\topp (T)$, with $\topp (T) \in \calu$, such that for all $s \in T$, we have $s\leq \topp (T)$.  A tree $T$ is a $j$-tree if $\omega_{\topp (T)}\cap \omega _{s,j} =\emptyset $.  Given a tree $T$, we will write $T_j$ to denote the maximal $j$-tree contained in $T$.
\end{defn}
Recall that $\omega _{s, 1}$ is the right half of $\omega _s$ and $\omega_{s,2}$ is the left half.  
The following definitions will help us organize our collections of tiles.  Recall that our vector field $v$ is defined on a set $E$; this set plays a role in the definitions of $\dense $ and $\udense$ below.  Similarly, the definitions of $\size$ depends on our other set $F$.

%
For $x\in \bbr ^2$, let $\chi  (x) = {{1} \over { 1+ |x|^{100}  } }$.  For any parallelogram $s$, let $\chi ^{(p)}_s$ be an $L^p$ normalized version of $\chi$ adapted to the parallelogram $s$.
\begin{defn}
Define the following for a parallelogram $s$ and a collection of parallelograms $\cals$:
\beqa
E_s &=& \{ (x,y)\in E \colon u(x) \in \omega _s \} \\
\dense(s) &=& \int _{ E_s } \chi ^{(1)} _s  \\
\udense(s) &=& \sup _{s' \geq s, s' \in \calu} \dense (s') \\
\size(\cals) &=& \sup _{\text{$1$-trees  }T\subseteq \cals}   \left( {1\over {|\topp(T)| } } \sum _{s\in T } 	| \langle \one _F , \varphi _s \rangle | ^2 \right) ^{1\over 2}.
\eeqa
\end{defn}
We remark that the function $\chi$ is needed for density since the wave packets $\varphi _s$ have Schwartz tails.  See the proofs of the tree and density lemmas.  The extra technicality involved in defining $\udense$ (as opposed to just $\dense$) is needed for our proof of the tree lemma (just as it is in the one-dimensional theory of \cite{LT}).  The cost is rather high:  a density estimate (see Estimate \ref{density} below) is still easily obtainable, but the maximal estimate becomes much more difficult to prove.  If 
$\udense (s) $ were equal to $\dense (s)$ for every tile $s$, then the tops of the trees constructed in Section \ref{organizationsection} are already prepared for an application of maximal technology.  Unfortunately this is not the case, and this difficulty prompts our consideration of the collections $\calr _j$ in Section \ref{maximalsection}.  See also the delicate sorting algorithm in Lacey-Li \cite{LL2}, where the authors wrestle with the same issue.


\section{Organization} \label{organizationsection}
In this section we carry out the main decomposition of the collection of tiles.  We sort a given collection of tiles into subsets of tiles of approximately constant density, and further into trees of approximately constant size.  
The relevance of trees is shown in the following:
\begin{lemma} [Tree lemma] \label{treelemma}
Let $T$ be a tree.  Suppose $\udense (T) \leq \delta $.  Suppose $\size (T) \leq \sigma $.
Then
\beqa
\sum _{s\in T}  | \langle \one _F, \varphi _s \rangle
			\langle \one _E , \phi _s \rangle | \lesssim \delta \sigma |\topp (T)|.
\eeqa
\end{lemma}
This is the ``Tree Lemma" from \cite{LL1}, which is the $2$-D version of the same in \cite{LT}.  We prove it in 
Section \ref{treesection}.  It reduces \eqref{model} to proving for each $0<\epsilon <1$
\beqa
\sum_{\delta }  \sum_{\sigma}  \sum_{T \in \calt _{\delta ,\sigma} }
	 \delta \sigma |\topp (T)| \lesssim |F|^{1- \epsilon} |E|^{\epsilon} .
\eeqa 
We can already prove this with the Estimates \ref{size}, \ref{density},  \ref{maximal} (appearing in the next lemma) and some bookkeeping -- this is carried out in Section \ref{optimization}.
\begin{lemma}[Organizational Lemma] \label{organizelemma}
Let $\cals $ be a finite collection of tiles.  Then there exist a partition of $\cals $ into trees $\calt_{\delta, \sigma}$ where $\delta , \sigma  $ are dyadic with $\delta \lesssim 1$, 
(i.e., $\cals = \bigcup _{\delta, \sigma} \bigcup _{T\in \calt_{\delta, \sigma}} T$) such that the following estimates hold:
\begin{estimate}\label{size} [Orthogonality]
\beqa
\sum_{T \in \calt _{\delta ,\sigma} } |\topp (T)|
	\lesssim {{|F| } \over {\sigma ^{2} } }.
\eeqa
\end{estimate}
\begin{estimate}\label{density} [Density]
\beqa
\sum_{T \in \calt _{\delta ,\sigma} } |\topp (T)|
	\lesssim {{|E|} \over {\delta }} .
\eeqa
\end{estimate}
\begin{estimate}\label{maximal} [Maximal]
For any $\epsilon >0$,
\beqa
\sum_{T \in \calt _{\delta ,\sigma} } |\topp (T)|
	\lesssim {{|F|^{1-\epsilon} |E|^{\epsilon} } \over {\delta  \sigma ^{1+ \epsilon} }}.
\eeqa
\end{estimate}
\end{lemma}
\begin{rem}
In fact we can take $\sigma \lesssim 1$, which we need (and prove) in the appendix.  
\end{rem}
In the remainder of this section we construct the collections of trees 
$\calt _{\delta, \sigma}$.  In the following sections we prove the estimates above.
Estimate \ref{size} follows from the construction of the trees $\calt _{\delta, \sigma}$, and the proof of the standard $\size$ lemma; we give a proof in Section \ref{sizesection}.
We prove Estimates \ref{density} and \ref{maximal} in Section \ref{maximalsection}.
We remark that we make these claims about the same family of trees.  This is in contrast to \cite{LT}, \cite{LL1}, \cite{LL2}, in which the argument has the form ``There exists a family $\calt_{size} $ such that $\cals _{\delta } = \cup _{T\in\calt _{size}} T$ and such that the size estimate holds for the collection $\calt _{size}$; further there is a (potentially different!) family $\calt _{density}$ such that $\cals _{\delta } = \cup _{T\in\calt _{density}} T$ and such that the density estimate holds for the collection $\calt _{density}$."


First, we sort the tiles by density:  
Let
\beqa
\cals _{\delta} = \{ s\in \cals \colon \udense (s) \in ({1\over 2}\delta , \delta ]\}
\eeqa
for dyadic $\delta$.  By the definition of $\dense$, we need only consider $\delta \leq ||\chi ||_1 \lesssim 1$.

We next sort each collection $\cals _{\delta} $ into families of trees with comparable size.  The following algorithm is a slight variant of the sorting algorithm used in \cite{LT} and in 
\cite{LL1}.  We want to ensure that $\topp (T) \in T$ for each tree $T$ in our construction.  There are some small technicalities that arise in the $2$-D situation due to the non-transitivity of the relation ``$\leq$".  
  Without loss of generality, we may assume our collection of tiles $\cals $ is finite, so we know there exists $\sigma _{max}$ such that $\size (\cals) \leq \sigma_{max}$ for every $T\subseteq \cals _{\delta }$.  This gives us a starting point for the following lemma.

\begin{lemma}\label{sizeiteration}
Let $\cals$ be a collection of tiles satisfying $\size (\cals) < \sigma $.   Then there exists a disjoint collection of trees $\calt _{\sigma}$ such that for all $T \in \calt _{\sigma}$, we have $\topp (T) \in T$, 
and 
\beqa
\size \left(\cals \setminus \bigcup _{T\in \calt _{\sigma } } T \right) 
<{\sigma \over 2}.
\eeqa
Finally, we have the estimate
\beqan \label{firstsizeestimate}
\sum_{T \in \calt _{\sigma}  } |\topp (T)| \lesssim {|F|\over \sigma ^2},
\eeqan
where here $F$ is the set used in the definition of $\size$.
\end{lemma}

\begin{rem}
Having $\topp (T) \in T$ will be helpful in Section \ref{maximalsection}.  See in particular the construction of the rectangles 
$R_T$ and the collections $\calt _R$.
\end{rem}

\begin{proof}
Initialize
\beqa
STOCK &=& \cals \\
\calt_{\sigma } &=& \emptyset .
\eeqa
In the following scheme we write $C$ to denote the constant used in the definition of tree (see Definition \ref{treedef}), which we assume is somewhat large.
While there is a $1$-tree $T \subseteq STOCK$ with 
\beqa
\sqrt{   {1\over \topp (T)} \sum _{s\in T} |\langle \one _f , \varphi _s \rangle |^2   }    \geq {{\sigma } \over C}
\eeqa
 and with $\topp (T) \in T$, choose $T$ with $c( \omega _{\topp(T)} )$ most clockwise, let $\tilde{T}$ be the maximal tree with top equal to $\topp (T)$, and update
\beqa
STOCK &:=& STOCK \setminus  \tilde{T}  \\
\calt_{\sigma } &:=& \calt_{\sigma } \cup \{ \tilde{T} \}.
\eeqa
(Again, we write $c( \omega _{\topp(T)} )$ to denote the center of $\omega _{\topp(T)} $.)  
\begin{rem}
We remark that our choice of $c(\omega _{\topp (T)})$ most clockwise will be used in the proof of Estimate \ref{firstsizeestimate} in Section \ref{sizesection}.  See specifically Claim \ref{technicalsizeclaim}.  
\end{rem}
When no such trees remain, we have the collection of trees 
$\calt_{\sigma}$ described in the statement of the lemma.  By construction we see that $\topp (\tilde{T})\in \tilde{T} $ and that 
$\size(\tilde{T} ) \geq {\sigma \over C}$ for each $\tilde{T} \in \calt_{\sigma}$.  The estimate \eqref{firstsizeestimate} follows rather standard arguments; we present the proof in Section \ref{sizesection}.  It remains to prove the following:
\begin{claim} 
\beqa
\size \left( STOCK \right) 
<{\sigma \over 2}.
\eeqa
\end{claim}
Consider a tree $T\subseteq STOCK$.  Without loss of generality, $T$ is a $1$-tree (since the definition of $\size$ only takes into consideration $1$-tree subtrees of $T$ anyway). We will partition $T$ into a collection $\calt _T$ of subtrees of $T$, each of which contains its top, as follows:  Initialize
\beqa
PANTRY &:=& T \\
T_{max} &:=& \emptyset .
\eeqa
While $PANTRY$ is nonempty, choose a tile $t$ of maximal length in $PANTRY$, let $T_t$ be the maximal subset of $PANTRY$ such that 
$s\leq t$ for $s\in T_t$, and update
\beqa
PANTRY & := & PANTRY \setminus T_t \\
T_{max} & := &T_{max} \cup \{t\}.
\eeqa
It is clear that this construction exhausts all of $T$; i.e., eventually $PANTRY$ becomes empty.  Since the tiles $t\in T_{max}$ all satisfy $\omega _{\topp(T)} \subseteq \omega _t$, and since each is maximal with respect to ``$\leq $", we know these tiles are pairwise disjoint.  On the other hand, they are all contained in $C \topp (T)$, and $t= \topp (T_t)$, so 
\beqa
\sum_{t\in T_{max}} |\topp (T_t)| \leq C |\topp (T)|.
\eeqa
Further, since each tree $T_t$ for $t\in T_{max}$ contains its top, we know 
\beqa
\sqrt{   {1\over \topp (T)} \sum _{s\in T} |\langle \one _f , \varphi _s \rangle |^2   }  \leq {\sigma \over C},
\eeqa
for otherwise $T_t$ would have been selected and put into $\calt_{\sigma}$.  Hence
\beqa
\sum _{s\in T} |\langle f, \varphi _s \rangle | ^2 
&=& \sum _{t\in T_{max} } \sum _{s \in T_t }|\langle f, \varphi _s \rangle | ^2 \\ 
&\leq & \sum _{t\in T_{max}}  | \topp(T_t) | {\sigma ^2 \over C^2} \\
&\leq & { \sigma ^2 |\topp (T)| \over C}.
\eeqa
This implies 
\beqa
\size(T) \leq {\sigma \over \sqrt{C} },
\eeqa
which proves the claim provided $C\geq 4$.
\end{proof}
By applying the lemma iteratively to each collection $\cals _{\delta}$, we obtain collections $\cals _{\delta , \sigma}$ and $\calt_{\delta, \sigma}$ such that 
\beqa
\cals _{\delta , \sigma} = \bigcup _{T\in \calt_{\delta, \sigma}} T
\eeqa
where the union is disjoint, such that $\udense (s)\sim \delta$ for $s\in \cals_{\delta, \sigma}$, and such that 
\beqa
\size(T) \sim \sigma \sim \sqrt{   {1\over \topp (T)} \sum _{s\in T} |\langle \one _f , \varphi _s \rangle |^2   } 
\eeqa
 for $T\in \calt_{\delta, \sigma} $.  This proves Lemma \ref{organizelemma}, except for Estimates \ref{density} and \ref{maximal}.  Note that Estimate  \ref{size} follows from \eqref{firstsizeestimate}.


\section{Main Lemmas} \label{lemmas}

Here we present the main lemmas needed to prove Estimates \ref{density} and \ref{maximal}.

\begin{lemma} \label{densitylemma}
Suppose $\calr$ is a collection of pairwise incomparable (under ``$\leq$") parallelograms of uniform width such that
$\dense (R)\geq \delta$ for $R\in \calr$.
Then
\beqa
\sum _{R \in \calr} |R| \lesssim {{|E|} \over {\delta} }.
\eeqa
\end{lemma}
Lemma \ref{densitylemma} is nothing more than the Density Lemma from \cite{LT} with straightforward modifications for the $2$-D setting.

\begin{lemma}\label{maximallemma}
Suppose $\calr$ is a collection of pairwise incomparable (under ``$\leq$") parallelograms of uniform width such that for each $R\in \calr$, we have
\beqan \label{densecond}
{{|E\cap u^{-1} (\omega _R) \cap R |} \over {|R|} } \geq\delta
\eeqan
and
\beqan \label{magnitudecond}
{1\over {|R|} }  \int _R \one _F \geq \lambda .
\eeqan
Then for each $\epsilon >0$,
\beqa
\sum_{R\in\calr} |R| \lesssim {{ |F| } \over {\delta \lambda ^{1+\epsilon} } }.
\eeqa
\end{lemma}
The proof of Lemma \ref{maximallemma} is contained in Section 3 of \cite{B1}.  More specifically, see estimate (3.10) on page 959, as well as the construction of the collection of parallelograms called $\calr _1$ there.
Note that this last lemma requires an assumption of the form
\beqa
{1\over {|R|} }  \int _R \one _F > \lambda ;
\eeqa
on the other hand, our assumption on $T\in \calt_{\delta, \sigma}$ is that $\size (T) \lesssim \sigma $ and
\beqa
\left( {1\over {|\topp(T)| } } \sum _{s\in T_1 } 	| \langle \one _F , \varphi _s \rangle | ^2 \right) ^{1\over 2} \gtrsim \sigma ,
\eeqa
where $T_1$ is the maximal $1$-tree in $T$.
The following lemma shows that the second kind of fact implies the first without much loss:
\begin{lemma}\label{intersection}
Let $F\subseteq \bbr ^2$.    Suppose $T$ is a tree with $\size (T) \lesssim \sigma $ and
\beqa
\left( {1\over {|\topp(T)| } } \sum _{s\in T_1 } 	| \langle \one _F , \varphi _s \rangle | ^2 \right) ^{1\over 2} \gtrsim \sigma ,
\eeqa
where $T_1$ is the maximal $1$-tree in $T$.
Then for any $\epsilon >0$,
\beqa
{{| \sigma ^{-\epsilon} \topp (T) \cap F|} \over {  |\sigma ^{-\epsilon} \topp (T) | }}
	\gtrsim \sigma ^{1+\epsilon}.
\eeqa
\end{lemma}
Lemma \ref{intersection} is proved in Section $\ref{intersectionsection}$;
it follows from $L^p$ and $BMO$-type estimates on a square function related to the notion of $\size$.

Estimate \ref{maximal} deserves more prominent mention.  An estimate in this spirit was proved in \cite{LL2}.  However here we have much better dependence on the parameter $\delta$ due to a rather simple observation.  The argument in \cite{LL2} follows essentially the argument of the density lemma, with an appeal to a maximal theorem to control $|\{ M_{\delta } \one _F > \lambda \} |$.  In our case of a vector field depending on only one variable, the relevant maximal operator was studied by the author in \cite{B1}, \cite{B2}.  However this approach is inefficient.  Instead of combining the density argument with a maximal function estimate (each of which costs in terms of ${1\over {\delta } }$), we appeal to an argument made in \cite{B1}, which directly estimates
\beqa
\sum_{R\in\calr} |R| \lesssim {{ |F| } \over {\delta  \lambda ^{1+\epsilon} } }
\eeqa
for any $\epsilon >0$.  In fact, this estimate was established en route to a covering lemma which implies the maximal theorem.  Interestingly, the improved $L^2$ estimates established in \cite{B2}, which interpolate to give improved $L^p$ estimates, are unhelpful in this setting, precisely because they are estimates on the operator norm, rather than on a sum like the one appearing immediately above.

\section{Balancing the estimates} \label{optimization}
In this section we carry out some computations which allow us to prove \eqref{model}, and hence the main theorem.  We now estimate
\beqa
\sum _{\delta} \sum_{\sigma}  \sum_{T \in \calt _{\delta ,\sigma} }
	 \delta \sigma  | \topp (T) |.
\eeqa
We have two cases.  Recall that $E$ and $F$ are sets with $|F|\leq |E|$.


\subsection{Case 1:  $\delta \geq {{|F|}\over {|E|}} $ }
A quick computation shows that (up to additive $O(\epsilon)$ terms in the exponents)
\begin{itemize}
   \item the maximal estimate is more efficient when $\sigma \geq {{|F|}\over {|E|}}$
  \item the density lemma is more efficient when $\sigma \leq {{|F|}\over {|E|}}$.
  \end{itemize}
\begin{rem}
The maximal estimate is more effective than the size estimate for $\delta \geq {{ |F|}\over {|E|}}$ and $\sigma$ close to ${{|F|}\over{|E|}}$.  Without this, we would not be able to obtain $L^p$ estimates for any $p<2$.
\end{rem}
For the first range, with $\delta $ fixed, we have for any $\epsilon >0$
\beqa
\sum_{\sigma \geq {{|F|}\over {|E|}} }  \sum_{T \in \calt _{\delta ,\sigma} }
	 \delta \sigma  | \topp (T) |
	& \lesssim & \sum_{\sigma \geq {{|F|}\over {|E|}} }
	\delta \sigma
	{{|F|^{1-\epsilon}|E|^{\epsilon} } \over {\delta \sigma ^{1+ \epsilon} } }  \\
	& = & |F|^{1-\epsilon}|E|^{\epsilon}  \sum_{\sigma \geq {{|F|}\over {|E|}} }  {1\over { \sigma ^{\epsilon} } } \\
	&\sim & |F|^{1-2\epsilon} |E|^{2\epsilon}.
\eeqa
Summing this over dyadic $1\gtrsim \delta \geq {{|F|}\over {|E|}}$ gives us a total of
$\lesssim |F|^{1-3\epsilon} |E|^{3\epsilon}$.

For the second range, with $\delta$ fixed, we have
\beqa
\sum_{{{|F|}\over {|E|}} \geq \sigma}  \sum_{T \in \calt _{\delta ,\sigma} }
	 \delta \sigma | \topp (T) |
	&\lesssim & \sum_{{{|F|}\over {|E|}} \geq \sigma}
	 \delta \sigma  {{|E|} \over {\delta } } \\
	& = & \sum_{{{|F|}\over {|E|}} \geq \sigma}
	 \sigma |E| \\
	&\sim & |F|.
\eeqa
Once again, summing this over dyadic $1\gtrsim \delta \geq {{|F|}\over {|E|}}$ gives us a total of
$\lesssim |F|^{1-\epsilon} |E|^{\epsilon}$.


\subsection{Case 2:  $\delta \leq  {{|F|}\over {|E|}} $ }

In this case, the size and density estimates alone will be enough for us.
A quick computation shows that 
\begin{itemize}
  \item The size estimate is most efficient when $\sigma \geq \sqrt{\delta {{|F|}\over {|E|}}  }$
  \item The density estimate is most efficient when $\sigma \leq \sqrt{\delta  {{|F|}\over {|E|}}  }$ .
\end{itemize}

We decompose our sum over $\sigma$ into these two ranges.  For the first range, we have
\beqa
\sum _{\sigma \geq \sqrt{\delta {{|F|}\over {|E|}}  } }
	\delta \sigma  {{|F|}\over {\sigma ^{2}} }
	&=& |F|\delta \sum _{\sigma \geq \sqrt{\delta {{|F|}\over {|E|}}  } }
		{1\over{\sigma  } } \\
	&\lesssim \sqrt{ |F||E|\delta }.
\eeqa
Summing over $\delta \leq {{|F|}\over {|E|}}$ gives us a total of
$\lesssim |F| \lesssim |F|^{1-\epsilon} |E|^{\epsilon} $, since $|F|\leq |E|$.

For the second range, we have
\beqa
\sum _{\sigma \leq \sqrt{\delta {{|F|}\over {|E|}} } } \delta \sigma
	{{|E|} \over {\delta } }
	&\sim &  |E| \sum _{\sigma \leq \sqrt{\delta {{|F|}\over {|E|}} } } \sigma  \\
	&\sim &\sqrt{ |F||E|\delta }.
\eeqa
Once again, summing over $\delta \leq {{|F|}\over {|E|}}$ gives us a total of
$\lesssim |F| \lesssim |F|^{1-\epsilon} |E|^{\epsilon} $, since $|F|\leq |E|$.

This completes the proof of the main estimate $\eqref{model}$ modulo the proofs of the lemmas, which are given in the following sections.


\section{Density lemma} \label{densitysection}
In this section we prove Lemma \ref{densitylemma}.  Let $\calr$ be as in the hypotheses of the lemma.  
For $k=0,1,2,\dots, $ let $\calr _k$ be the collection of $R \in \calr$ such that
\beqa
| u^{-1} (\omega _R) \cap 2^k R \cap E | \geq {1\over {100}} \delta 2^{20k} |2^k R| ,
\eeqa
and such that $k$ is the least integer with this property.  Note $\calr = \cup _k \calr _k$, since if $R\in \calr$ but 
$R\not\in \cup_k \calr _k$, then 
\beqa
\dense (R) & \leq & \int _{E_R } \chi _R ^{(1)} \\ 
& \leq & \sum _{k=0} ^{\infty} | u^{-1} (\omega _R) \cap 2^k R \cap E | 2^{-100k} {1\over |R|} \\
& \leq & {1\over 100} {\delta \over |R|} \sum _{k=0} ^{\infty} 2^{25k} |R| 2^{-100k} \\
& \leq & {\delta \over 50}.
\eeqa

We now run an iterative selection procedure to find a subset of $\calr _k$ such that the parallelograms $2^k R$ are disjoint:

Initialize
\beqa
STOCK &=& \calr _k \\
\widetilde{\calr _k} &=& \emptyset .
\eeqa
While $STOCK \neq \emptyset$, choose $R$ with maximal length, let
\beqa
\cala _R = \{ R' \in STOCK
	\colon 2^k R' \cap 2^k R \neq \emptyset \text{  and  } \omega_{R'} \cap \omega _R \neq \emptyset \},
\eeqa
and update
\beqa
STOCK: &=& \calr _k \setminus \cala _R \\
\widetilde{\calr _k} &=& \widetilde{\calr _k}\cup \{R\}.
\eeqa
Note that the parallelograms in $\cala _R$ are pairwise disjoint by the pairwise incomparability of parallelograms in $\calr$, 
and because
$ \omega_{R'} \cap \omega _R \neq \emptyset $  for $R' \in \cala _R$.  Hence, using the definition of $\calr _k$, we have
\beqa
\sum _{R\in \calr _k} |R| &=& \sum _{R \in \widetilde{\calr _k} } \sum _{R' \in \cala _R} |R'| \\
& \lesssim & 2^{2k} \sum _{R \in \widetilde{\calr _k} }|R| \\
&\lesssim & 2^{2k} 2^{-20k} {1\over \delta} \sum_{R \in \widetilde{\calr _k} } | u^{-1} (\omega _R) \cap 2^k R \cap E | \\
&\lesssim & 2^{-18k} {1\over \delta} |E|,
\eeqa
where in the last inequality we have used the fact that the parallelograms $2^k R$ are pairwise incomparable, and that $\omega _R = \omega _{2^k R}$, so that the sets $\{ u^{-1} (\omega _R) \cap 2^k R \}$ are disjoint.  Finally, we sum over $k$ to obtain the result.

\section{Proofs of maximal and density estimates} \label{maximalsection}

We now look more closely at the collections $\calt_{\delta, \sigma}$.  For the remainder of this section we regard $\delta$ and $\sigma$ as fixed.  Notation in this section is understood to depend on both $\delta$ and $\sigma$.
(So, for example, $\calt = \calt _{\delta, \sigma}$.)  We begin by isolating a collection of tiles with density $\delta$.  First, let
\beqa
\tilde{\calr} = \{ R\in \calu \colon \dense (R) \sim \delta \}.
\eeqa
We now find a maximal subset of $\tilde{\calr}$ whose elements are pairwise incomparable.  Initialize:
\beqa
STOCK &=& \tilde{\calr} \\
\calr &=& \emptyset.
\eeqa
While $STOCK \neq \emptyset$, choose $R$ of maximal length in $STOCK$.  Define
\beqa
\cala _R = \{ R' \in STOCK \colon R' \leq R \},
\eeqa
and update
\beqa
STOCK &=& STOCK \setminus \cala _R \\
\calr &=& \calr \cup \{R\}.
\eeqa
When the loop terminates, elements of $\calr$ are pairwise incomparable (under $\leq$), and $\calr$ is maximal with respect to this property.  
\begin{rem}
Recall that for $T \in \calt$, $\udense (\topp (T)) \sim \delta$, but maybe $\dense (\topp (T))$ is much less than $\delta $.  This makes the maximal Lemma \ref{maximallemma} unavailable to us.  Note that several ingredients are required, and $\topp (T)$ may lack the $\dense$ required.  The work in this section goes to organizing the trees in such a way that we can legitimately appeal to Lemma \ref{maximallemma}.
\end{rem}
Next we associate to each tree
$T\in \calt $ a parallelogram $R_T \in \calr$.  This requires a few steps.
Note that for each $s\in \cup _{T\in \calt} T$, we have $\udense (s)\sim \delta$.  By Lemma \ref{sizeiteration}, we know that $\topp (T) \in T$ for each 
$T\in \calt$.  Hence $\udense ( \topp (T) ) \sim \delta$.  This means there exists a parallelogram $\tilde{R} \in \tilde{\calr}$ such that $\dense (R) \sim \delta$ and such that 
$\topp (T) \leq \tilde{R}$.  (This is the reason why it is convenient to have 
$\topp (T) \in t$.)  Further, for each $\tilde{R}\in \tilde{\calr}$, there is $R \in \calr$ (again, possibly not unique) such that $\tilde{R} \leq R$.  Hence we may assign to each $T\in \calt$ some $R\in \calr$, and there is $\tilde{R}$ such that $\topp (T) \leq \tilde{R} \leq R$.  (Of course there may be more than one $R$ to choose from for each $T$; choose one!)  Call this parallelogram $R_T$.  Now for each $R\in \calr$, define
\beqa
\calt _R = \{ T \in \calt \colon R_T = R \}.
\eeqa
By construction,
\beqa
\calt  = \cup _{R\in\calr} \calt _R.
\eeqa
Our goal now is to control
\beqa
\sum _{R\in\calr} \sum _{T\in \calt _R} |\topp (T)| .
\eeqa
First, we'll show that for all $R \in \calr$, 
\beqa
\sum _{T\in \calt _R} |\topp (T)| \lesssim |R|.
\eeqa
The collection $\{ \topp (T) \colon T\in \calt _R \}$ need not be pairwise disjoint, but we do have the following satisfactory substitute. 
\begin{claim} \label{disjointtops}
There exists $\overline{\calt _R} \subseteq \calt_R $ such that $\{\topp (T) \colon T\in \overline{\calt _R} \}$ is pairwise disjoint and such that
\beqa
\sum _{T\in\calt_R} |\topp (T)| \lesssim \sum _{T\in \overline{ \calt_R} } |\topp (T)| .
\eeqa
\end{claim}

\begin{proof}

Initialize
\beqa
STOCK &=& \calt _R \\
\overline{\calt _R } &=& \emptyset .
\eeqa
While $STOCK \neq \emptyset$, choose $T \in STOCK$ such that $\topp (T)$ is of maximal length.  Then define
\beqa
\cala _T =\{ T' \in STOCK \colon \topp (T' ) \cap \topp (T) \neq \emptyset \} ,
\eeqa
and update
\beqa
STOCK &:=& STOCK \setminus \cala _T \\
\overline{\calt _R } &:=& \overline{\calt _R  } \cup \{ T \} .
\eeqa
We stop when $STOCK$ is empty.  By construction, the tops of the trees in $\overline{\calt _R }$ are pairwise
disjoint.  Now we show that
\beqa
\sum _{T' \in \cala _T } |\topp (T') | \leq C' |\topp (T)|.
\eeqa
With this we'll know that
\beqa
\sum _{T \in \calt _R} |\topp (T)| =
\sum_{T\in \overline{\calt _R}} \sum _{T' \in \cala _T} |\topp (T')| \leq C' \sum _{T\in \overline{\calt _R}} |\topp (T) |.
\eeqa
Suppose not.  Define $S = \cup _{T' \in \cala _T }  T_1 ' $, where for a tree $T$, define $T_1$ to be the maximal $1$-tree contained in $T$. 
We claim $S$ can be partitioned into a small number of trees $S_j$, $j=1, \dots, 10C^2$,  with each a $1$-tree.  To see that they are $1$-trees, suppose $s\in T' \in \cala _T$.  Then $\omega _{s,2} \supseteq \omega _{\topp(T')} \supseteq \omega _{\topp(T)}$, so $\omega_{s,1} \cap \omega_{\topp (T)} = \emptyset$.  To see that we only need a few trees, just note that for each $T' \in \cala _T$, 
$\topp (T')\subseteq C (\topp (T))$.  Then since each $s\in T'$ satisfies $s\subseteq  C (\topp (T'))$, we know that $S$ can be partitioned into $\sim C^2$ subtrees $S_j$ by considering (possibly overlapping) tiles in 
$C^2\topp (T )   $ of height $ w$ and length the same as length of $\topp (T)$.  
%
Hence
\beqa
\sum _{j = 1} ^{10 C^2} \sum _{s\in S_j} | \langle f, \varphi _s \rangle | ^2
&\geq & \sum _{T' \in \cala _T} \sum _{s\in T' _1 } | \langle f, \varphi _s \rangle | ^2 \\
	&\geq & {1\over 4} \sum _{T' \in \cala _T} \sigma ^2 |\topp (T') | \\
	&\geq & \sigma ^2 {{C'} \over 4} |\topp (T)| \\
\eeqa
Provided $C'$ is taken large enough (with respect to a universal constant $C$ mentioned in Section \ref{defs}), one of the trees $S_j$ satisfies
 $\size (S_j) \geq 10\sigma $, which is impossible since the trees $T\in \calt _R$ were chosen from a collection with $\size $ less than $\sigma$.  This proves the second claim about $\overline{\calt _R}$.
\end{proof}


\subsection{Proof of the density estimate}
We are already in position to prove Estimate \ref{density}.  Note that the collection $\calr$ constructed above is of pairwise incomparable parallelograms of uniform width and $\dense \sim \delta$.  Hence the previous claim, together with Lemma \ref{densitylemma}, implies
\beqa
\sum _{R\in\calr} \sum _{T\in\calt_R} |\topp (T) | & \lesssim & \sum _{R\in\calr} |R| \\
& \lesssim &{{|E|}\over {\delta } }.
\eeqa


\subsection{Proof of the maximal estimate}
The proof of Estimate \ref{maximal} is a bit more involved. 
For the rest of this section, fix $\epsilon >0$. The first key step is to sort the parallelograms in $\calr$ by how heavily they are covered by the trees in $\calt_R$.  Specifically:  for integers $j\geq 0$, define
\beqa
\calr _j = \{ R\in\calr \colon \sum _{T\in\calt _R } | \topp (T)| \sim 2^{-j} |R| \}.
\eeqa
Since our goal is to control 
\beqa
\sum _{R\in \calr} \sum _{T \in \calt_R } |\topp (T)| 
\sim \sum _j \sum _{R\in \calr _j} \sum _{T\in \calt _R} |\topp (T)|
\sim \sum _j 2^{-j} \sum _{R\in \calr _j} |R|,
\eeqa
it is enough to estimate $\sum _{R\in \calr _j} |R|$ with suitable dependence on $j$.

In order to apply maximal technology (in the form of Lemma \ref{maximallemma}), we must find parallelograms $R$ that heavily intersect $F$, and that also contain a large subset on which $v$ points in the direction of $R$.  Because of the Schwartz tails in the definition of $\dense$, we do not know that each
$R\in\calr_j$ satisfies
\beqa
| u^{-1} (\omega _R) \cap E \cap R | \gtrsim \delta |R|.
\eeqa
Rather, we know that
\beqan \label{kcondition}
| u^{-1} (\omega _{R}) \cap E \cap 2^k  R | \gtrsim 2^{20k} \delta |R|
\eeqan
for some integer $k\geq 0$, as in Section \ref{densitysection}.  Define $\calr _{j,k}$ to be the set of $R\in \calr _j$ such that condition \eqref{kcondition} holds for $R$ but such that it does not hold with any smaller $k$.
Similarly, we cannot conclude that $R$ itself intersects $F$ heavily.  Recall that Lemma \ref{intersection} guarantees that $F$ intersects $\sigma ^{-\epsilon}\topp (T)$ heavily, whenever $T \in \calt _{\delta ,\sigma }$; we cannot however, conclude that $F$ intersects $\topp (T)$ itself.  This causes some minor differences in the treatment of the cases $2^k \geq \sigma ^{-\epsilon}$ and
$2^k \leq \sigma ^{-\epsilon}$ that the reader should not take too seriously.
It suffices then to control sums like
\beqa
\sum _{R\in\calr_{j,k}} |R|
\eeqa
with suitable dependence on $k$ and $j$.


\subsubsection{Case1:  $2^k \geq \sigma ^{-\epsilon}$}
We want to apply Lemma \ref{maximallemma} to the collection $\calr _{j,k}$.  The defining condition of $\calr _{j,k}$ gives us the kind of information needed by the hypothesis \eqref{densecond}.  The following claim gives us the kind of information needed by the hypothesis \eqref{magnitudecond}.
\begin{claim} \label{hard}
For $R \in \calr _{j,k}$
\beqa
{{ |F \cap 2^k R | } \over  {|2^k R|} } \gtrsim 2^{-j} \sigma ^{1+3\epsilon}
				\left( {{\sigma ^{-\epsilon} } \over {2^k} } \right) ^2 .
\eeqa
\end{claim}
We postpone the proof of the claim until the end of this section.
With the claim, the only ingredient still needed to apply Lemma \ref{maximallemma} is the pairwise incomparability of the parallelograms in question.  We arrange this with the usual type of sorting algorithm.
Initialize
\beqa
STOCK &=& \calr_{j, k} \\
\widetilde{\calr_{j, k}} &=& \emptyset .
\eeqa
While $STOCK \neq \emptyset$, choose $R$ with maximal length, let
\beqa
\cala _R = \{ R' \in STOCK
	\colon 2^k R' \cap 2^k R \neq \emptyset \text{  and  } \omega_{R'} \cap \omega _R \neq \emptyset \},
\eeqa
and update
\beqa
STOCK: &=& \calr_{j, k} \setminus \cala _R \\
\widetilde{\calr_{j, k}} &=& \widetilde{\calr _{j,k} }\cup\{R\}.
\eeqa
(Note $\omega _R = \omega _{CR}$ for any $C$.)  Since the parallelograms $R' \in \cala _R$ are pairwise incomparable, we know they are in fact disjoint (see earlier in Section \ref{maximalsection} for a similar argument), so
\beqa\sum _{R' \in \cala _R} |R'| \lesssim |2^k R|.
\eeqa
Hence
\beqa
\sum_j \sum _k \sum _{R \in \calr_{j, k} } \sum _{T \in \calt _R} |\topp (T)|
&\lesssim &
\sum_j \sum _k \sum _{R \in \calr_{j, k} }  2^{-j}|R| \\
&\lesssim &
\sum_j \sum _k \sum _{R \in \widetilde{\calr_{j, k}} } \sum _{R' \in \cala _R} 2^{-j} |R'| \\
&\lesssim & \sum_j \sum _k \sum _{R \in \widetilde{\calr_{j, k}} }  2^{-j}  | 2^k R| .
\eeqa
We now focus our attention on
\beqa
\sum _{R \in \widetilde{\calr_{j, k}  } }  2^{-j}  | 2^k R|.
\eeqa
Claim \ref{hard} together with the defining condition for parallelograms in $\calr _{j,k}$ allows us to apply Lemma \ref{maximallemma}, with ``$\delta$" in  \eqref{densecond} being $2^{20k}\delta$ and ``$\lambda$" in \eqref{magnitudecond} 
being $2^{-j}2^{-2k} \sigma ^{1+ O(\epsilon)}$, as in Claim \ref{hard}.  The huge gain in $k$ from \eqref{kcondition} allows us to sum the contributions from the various $\calr_{j,k}$.  More specifically, Lemma \ref{maximallemma} yields
\beqa
\sum _{R\in\widetilde{\calr_{j,k}} } |2^k R|  \lesssim  {1\over {2^{20k} \delta }} 
	{{|F|}\over { (\sigma ^{1+\epsilon } 2^{-2k}2^{-j})^{1+\epsilon} } }
\eeqa
This obviously sums in $k$ to prove
\beqa
\sum _{R\in \calr _{j} } \sum _{T \in \calt _R}  |\topp (T)| \lesssim
\sum _{R\in\calr_{j}} 2^{-j}| R|  \lesssim  {1\over {\delta }} {{2^{\epsilon j} |F|}\over { (\sigma ^{1+\epsilon } )^{1+\epsilon} } };
\eeqa
this estimate is effective for small $j$.
Estimate \ref{density} tells us that for any $j$,
\beqa
\sum _{R\in \calr _{j} } \sum _{T \in \calt _R}  |\topp (T)| \lesssim 
\sum _{R\in\calr_{j}} 2^{-j}| R| \lesssim  2^{-j}{{|E|}\over {\delta }};
\eeqa
this estimate is effective for large $j$.
It remains to balance these two estimates:
\beqa
\sum_{j\geq 0} \sum _{R \in \calr_j } \sum _{T \in \calt _R} |\topp (T)|
&=& 	\sum _{j \leq \log {{ |E|\sigma } \over {|F|} } }
		\sum _{R \in \calr_j } 2^{-j} |R|
	+ \sum _{j \geq \log {{ |E|\sigma } \over {|F|} } }
		\sum _{R \in \calr_j } 2^{-j} |R|  \\
&\lesssim & \sum _{j \leq \log {{ |E| \sigma } \over {|F|} } }
		 2^{\epsilon j}
		{{|F|} \over {\delta \sigma ^{(1+\epsilon) ^2 }  } }
	+ \sum _{j \geq \log {{|E| \sigma } \over {|F|} } }
		 2^{-j}  {{|E|} \over {\delta } }  \\
&\lesssim &   {{|F|^{1-\epsilon }|E|^{\epsilon} } \over {\delta \sigma ^{(1+\epsilon) ^2 } } }\\
&\lesssim & {{|F|^{1-5\epsilon }|E|^{5\epsilon} } \over {\delta \sigma ^{1 + 5 \epsilon } } },
\eeqa
\begin{rem}
Of course the first sum above is empty when $\sigma \leq {{|F|} \over {|E|}}$; in this case we recover the density estimate.  Recalling Section \ref{optimization}, we see that in this range of $\sigma$ we have no need for the maximal estimate anyway.
\end{rem}
This completes the proof of the maximal estimate, except for the proof of Claim \ref{hard}, which we turn to now.  
\begin{proof}[Proof of Claim \ref{hard}]
For each $T \in \overline{\calt_R}$, Lemma \ref{intersection} tells us that
\beqa
{{| \sigma ^{-\epsilon} \topp (T) \cap F |} \over {|\sigma ^{-\epsilon} \topp (T)|} } \geq \sigma ^{1+ \epsilon} .
\eeqa
One minor technical problem is that the parallelograms $\sigma ^{-\epsilon} \topp (T)$ might not be disjoint.  But since all parallelograms $\{ \topp (T) \colon T \in \overline{ \calt _R } \}$ have (essentially) the same orientation, we may use a standard covering argument to select a subset $\widetilde{\calt _R} $ of $\overline{\calt _R}$ such that
\beqa
\{ \sigma ^{-\epsilon } \topp (T) \} _{T \in \widetilde{\calt _R} }
\eeqa
is pairwise disjoint, and such that
\beqa
| \bigcup _{T \in \widetilde{\calt _R} } \sigma ^{-\epsilon } \topp (T) |
 \gtrsim  | \bigcup _{T \in \overline{\calt _R} } \sigma ^{-\epsilon } \topp (T) |.
\eeqa
Hence
\beqa
|F \cap  C \sigma ^{-\epsilon } R | & \gtrsim &
	| \bigcup _{T \in \widetilde{\calt _R} } \sigma ^{-\epsilon } \topp (T) \cap F | \\
 & = & \sum _ {T \in \widetilde{\calt _R} } |\sigma ^{-\epsilon } \topp (T) \cap F |  
\hspace{1cm} \text{     by disjointness}  \\
& \gtrsim & \sigma ^{1+ \epsilon} \sum _ {T \in \widetilde{\calt _R} } |\sigma ^{-\epsilon } \topp (T)  | 
\hspace{1cm} \text{    by Lemma \ref{intersection} }\\
& \gtrsim & \sigma ^{1+ \epsilon} | \bigcup _{T \in \widetilde{\calt _R} } \sigma ^{-\epsilon } \topp (T) | \\
& \gtrsim & \sigma ^{1+ \epsilon} | \bigcup _{T \in \overline{\calt _R} } \sigma ^{-\epsilon } \topp (T) | \\
& \gtrsim & \sigma ^{1+ \epsilon} | \bigcup _{T \in \overline{\calt _R} } \topp (T) | \\
& \gtrsim & \sigma ^{1+ \epsilon} \sum _{T \in \overline{\calt _R} } |\topp (T) | 
\hspace{1cm} \text{    by disjointness}\\
& \gtrsim & \sigma ^{1+ \epsilon} \sum _{T \in \calt _R } |\topp (T) | 
\hspace{1cm} \text{    by Claim \ref{disjointtops} }\\
& \gtrsim & \sigma ^{1+ \epsilon} 2^{-j} |R|
\hspace{1cm} \text{    by definition of $\calr _j$}.
\eeqa
This finishes the proof of Claim $\ref{hard}$.  
\end{proof}


\subsubsection{ Case 2:  $2^k \leq \sigma ^{-\epsilon}$}
This section is very similar to the previous section.
As in the last section, we verify the hypotheses of Lemma \ref{maximallemma} for a suitable collection.

We consider all of these collections $\calr _{  j,k}$ together.  Let
\beqa
\calr _{ j, small} =\bigcup _{0\leq k\leq \log \sigma ^{-\epsilon} } \calr _{j,k} .
\eeqa
Now we sort the tiles as before:
Initialize
\beqa
STOCK &=& \calr _{ j, small} \\
\widetilde{\calr _{ j, small}} &=& \emptyset .
\eeqa
While $STOCK \neq \emptyset$, choose $R$ with maximal length, let
\beqa
\cala _R = \{ R' \in STOCK
	\colon \sigma ^{-\epsilon} R' \cap \sigma ^{-\epsilon} R \neq \emptyset \text{  and  } \omega_{R'} \cap \omega _R \neq \emptyset \},
\eeqa
and update
\beqa
STOCK: &=& \calr _{small} \setminus \cala _R \\
\widetilde{\calr _{ j, small}} &=& \widetilde{\calr _{ j, small}}\cup \{R\}.
\eeqa
As before, we have
\beqa
\sum _{R\in \calr_{ j, small} } |R|
	\leq \sum _{R\in \widetilde{\calr _{ j, small}}  }  \sum _{R' \in \cala_R} |R'|
	\leq \sum _{R\in \widetilde{\calr _{ j, small}}  } |\sigma ^{-\epsilon} R |.
\eeqa
We again note several properties of the parallelograms in $\widetilde{\calr _{ j, small}}$.  First, they are pairwise incomparable.  Second, they satisfy the estimate
\beqa
{{ |\sigma ^{-\epsilon} R \cap E \cap u^{-1} (\omega _{\sigma ^{-\epsilon} R} )| } \over {|\sigma ^{-\epsilon} R| } } 
	\gtrsim \sigma ^{2\epsilon} \delta .
\eeqa
This gives us the density estimate
\beqan \label{densityest2}
\sum _{R\in \widetilde{\calr _{ j, small}}  }   |\sigma ^{-\epsilon} R|
	\lesssim {{|E|} \over {\sigma ^{2\epsilon} \delta } },
\eeqan
from a direct application of Lemma \ref{densitylemma}.
Third, just as in Claim \ref{hard}, they satisfy the estimate
\beqa
{{ | \sigma ^{-\epsilon} R \cap F| } \over { |\sigma ^{-\epsilon} R|} } \gtrsim 2^{-j} \sigma ^{1+\epsilon}.
\eeqa
So by Lemma \ref{maximallemma}, we have 
\beqan \label{maximalest2}
\sum _{R\in \widetilde{\calr _{ j, small}}  } |\sigma ^{-\epsilon} R |
&\lesssim &  { {|F}| \over { \delta\left(2^{-j} \sigma ^{1+\epsilon} \right) ^{1+\epsilon}  } }.
\eeqan
As before, we split the sum into large and small $j$ and use $\eqref{densityest2}$ and $\eqref{maximalest2}$, respectively:
\beqa
\sum_{j\geq 0} \sum _{R \in \calr_{ j, small} } \sum _{T \in \calt _R} |\topp (T)|
&=& 	\sum _{j \leq \log {{ |E|\sigma } \over {|F|} } }
		\sum _{R \in \calr_{ j, small} } 2^{-j} |R| \\
	&+& \sum _{j \geq \log {{ |E|\sigma } \over {|F|} } }
		\sum _{R \in \calr_{ j, small} } 2^{-j} |R|  \\
&\lesssim & \sum _{j \leq \log {{ |E|\sigma } \over {|F|} } }
		 2^{\epsilon j}
		{{|F|} \over {\delta \sigma ^{(1+ \epsilon)^2 } } }  \\
	&+& \sum _{j \geq \log {{ |E| \sigma } \over {|F|} } }
		 2^{-j}  {{|E|} \over {\sigma ^{2\epsilon}  \delta } }  \\
&\lesssim &  {{|F|^{1-5\epsilon }|E|^{5\epsilon} } \over {\delta \sigma ^{1+5\epsilon } } },
\eeqa
which is what we needed, since $\epsilon$ is arbitrary.


\section{Large size implies large intersection with $F$} \label{intersectionsection}
\begin{rem}
The title of the section is technically a bit misleading, since $\size (T)$ is actually the supremum over all subtrees of $T$ of an 
$l^2$-type norm; nevertheless, the trees obtained through the selection procedure in Section \ref{organizationsection} all satisfy the property that the full tree (essentially) achieves this supremum.
\end{rem}
To prove Lemma \ref{intersection}, we need the following notation.  For a fixed $1$-tree $T$, define the operator
\beqa
\Delta (f) = \left( \sum_{s\in T} |\langle f , \varphi _s \rangle| ^2 {{\one _s} \over {|s|}}  \right)^{1\over 2}.
\eeqa
We need the following facts about $\Delta$.
\begin{lemma} \label{lpsquarefcn}
For any $N> 0$, we have
\beqa
||\Delta f ||_p \lesssim  || f \beta _{N,T} || _p
\eeqa
for $p\in (1, \infty) $, where
\beqa
\beta _N  (x_1, x_2) = {1\over {1+|x_1|^N +|x_2|^N } },
\eeqa
and
$\beta  _{N, T}$ is an $L^{\infty}$-normalized version of $\beta _N $ adapted to $\topp (T)$.  The implicit constant depends on $N$ but not on $T$.
\end{lemma}
We prove Lemma \ref{lpsquarefcn} in Section \ref{sqfcnsection}.
Of course proving $||\Delta f || _2  \lesssim ||f|| _2 $ is straightforward; indeed, it is an easy special case of Lemma \ref{bessel}.  The work is in inserting the smooth cutoff $\beta _N$, which is the point of Lemma \ref{bessel}, and moving below $L^2$.
Second,
\begin{lemma} \label{bmosquarefcn}
\beqa
||\Delta f ||_2 \lesssim  {1\over { |\topp (T) | ^{1\over 2 } } } \int_{C\topp (T)} \Delta f,
\eeqa
provided that $T$ satisfies the following uniform size estimate:
\beqa
\sup _{\text{$1$-trees} \hspace{.2cm}  T' \subseteq T }
	\left( {1\over {|\topp (T')| } } \sum _{s\in T'}
	|\langle f , \varphi _s \rangle | ^2  \right) ^{1\over 2}
\lesssim \left( {1\over {|\topp (T)| } } \sum _{s\in T}
	|\langle f , \varphi _s \rangle | ^2  \right) ^{1\over 2} .
\eeqa
\end{lemma}
The condition in the last lemma is the one mentioned in the remark at the beginning of this section.
We prove Lemma \ref{bmosquarefcn} in Section \ref{bmosection}.
The point of these lemmas is that $||\Delta f ||_2$ is closely related to $\size(T)$.  Indeed, 
\beqa
||\Delta f ||_2 ^2 = \sum _{s\in T} |\langle f , \varphi _s \rangle | ^2 .
\eeqa
On the other hand, we want information about $|F\cap \topp (T)|$ (or possibly $|F \cap M\topp (T)|$ for a dilate $M\topp (T)$ of $\topp (T)$, which is actually what we will obtain below), which is much more closely related to $||\Delta f||_p$ for $p$ close to $1$, as we see below.  
Combining these two lemmas and H\"{o}lder's inequality gives us
\beqa
 \left(  {1\over { |\topp (T) | } }\sum_{s\in T} |\langle f , \varphi _s \rangle| ^2 \right)^{1\over 2}
&=& {1\over { |\topp (T) | ^{1\over 2} } } ||\Delta f ||_2 \\
&\lesssim &  {1\over { |\topp (T) |  } } \int _ {C\topp (T)} \Delta f \\
&\lesssim & \left(  {1\over { |\topp (T) |  } }  \int (\Delta f) ^{1+\epsilon} \right)  ^{1\over {1+\epsilon} } \\
&\lesssim & \left(  {1\over { |\topp (T) |  } }
	\int (f \beta _{N, T} ) ^{1+\epsilon} \right)  ^{1\over {1+\epsilon} } . \\
\eeqa
Applying this with $f= \one _F$ and a tree $T$ such that 
$ \left(  {1\over { |\topp (T) | } }\sum_{s\in T} |\langle f , \varphi _s \rangle| ^2 \right)^{1\over 2} \sim \sigma $ 
gives us for any $N$,
\beqa
\sigma ^{1+\epsilon}|\topp(T)| &\lesssim & \int \one _F (\beta _{N, T} )^{1+\epsilon}  \\
& \lesssim  & |\sigma ^{-\epsilon} \topp (T) \cap F| + \sigma ^{(N-2) \epsilon} | \sigma ^{-\epsilon} \topp(T) |
\eeqa

This proves Lemma \ref{intersection} since $N$ can be chosen arbitrarily large with respect to $\epsilon$.


\section{Proof of Tree Lemma} \label{treesection}
In this section we present a proof of Lemma \ref{treelemma}.  Recall that we have a fixed tree $T$ in mind.  For notational convenience we assume that the slope of the long side of $\topp (T)$ is zero.  We write $\pi _1 (E), \pi _2 (E)$ to denote the vertical, horizontal (respectively) projections of a set $E$. 
Of course the width of every tile in $T$ is a fixed number $w$.  Let $\calj_1$ be a partition of $\bbr$ (the horizontal axis) into dyadic intervals such that $3J \times \bbr$ does not contain any tile $s\in T$, and such that $J$ is maximal with respect to this property.  Now let $\calj _2$ be a partition of $\bbr$ (the vertical axis) into intervals of width ${1\over 3} |\pi _2 (\topp (T))| $.  Let
\beqa
 \calp = \bigcup _{J_1 \in \calj_1 } \bigcup _{J_2 \in \calj _2} J_1 \times J_2.
\eeqa
  This is a partition of $\bbr ^2$.  The parallelograms $P\in
 \calp$ are the smallest relevant parallelograms for this tree.  The parallelograms $P\in \calp $ with $\pi _1 (P)$ far away from $\topp (T)$ are defined so as to still be able to take advantage of the density estimate for tiles in $T$.  Now for each $P\in\calp$ we split the operator $L$ into two pieces, one corresponding to tiles with larger $x$-projection than $P$, the other to tiles with smaller $x$-projection than $P$:  let
\beqa
T_P ^+ &=& \{ s\in T \colon |\pi _1 (s)| > |\pi _1 (P) | \}    \\
T_P ^- &=& \{ s\in T \colon |\pi _1 (s)| \leq |\pi _1 (P) | \}   \\
L_P ^+ &=& \sum _{s\in T_P ^+ } \langle f , \varphi _s \rangle \phi _s \one _E \\
L_P ^- &=& \sum _{s\in T_P ^- } \langle f , \varphi _s \rangle \phi _s \one _E .
\eeqa

Note that for appropriate $\epsilon _s$ with $|\epsilon_s| =1 $, we have
\beqan \label{star} \nonumber
\sum _{s\in T} |\langle f, \varphi _s \rangle \langle \phi _s \one _E \rangle | &=&
\sum _{s\in T} \epsilon _s  \langle f, \varphi _s \rangle \langle \phi _s \one _E \rangle \\ \nonumber
 &=& \int \sum _{s\in T} \epsilon _s  \langle f, \varphi _s \rangle  \phi _s \one _E \\ \nonumber
&=& \sum _{P \in \calp }  \int_P \sum _{s\in T} \epsilon _s  \langle f, \varphi _s \rangle  \phi _s \one _E \\
&=&    \sum _{P \in \calp }  \int_P L_P ^-  + \sum _{P \in \calp }  \int_P L_P ^+.
\eeqan
The main term will come from parallelograms $P \in \calp$ close to $\topp (T)$; estimates on parallelograms $P$ away from $\topp (T)$ will come with a decay factor.  To make things more precise, define for $k\geq 1$, 
\beqa
\calp _0 &=& \{ P \in \calp \colon {{ \dist (\pi _2 (P), \pi _2 (\topp (T)) )} \over {|\pi _2 (\topp (T)) | }}  \leq  1 \} \\
\calp _k &=& \{ P \in \calp \colon {{ \dist (\pi _2 (P), \pi _2 (\topp (T)) )} \over {|\pi _2 (\topp (T)) | }} \in (2^{k-1} ,2^{k} ] |  \}.
\eeqa
We focus first on the first term in \eqref{star}.  To control it we need only spatial decay in both the horizontal and vertical directions.

\subsection{Small tiles}
For notational convenience, we further consider for $l\geq 1$, 
\beqa
\calp _{k, 0} &=& \{ P \in \calp _k \colon {{ \dist (\pi _1(P), \pi _1( \topp (T))  )} \over { |\pi _1( \topp (T))| }} \leq 1 \}, \\
\calp _{k, l} &=& \{ P \in \calp _k \colon {{ \dist (\pi _1(P), \pi _1( \topp (T))  )} \over { |\pi _1( \topp (T))| }}
		\in (2^{l-1}, 2^l ] \}.
\eeqa
We divide the sum in the definition of $L_P ^{-}$ into pieces according to how large the tiles are.  Specifically, let 
\beqa
T_j = \{s \in  T_P ^- \colon |s| = 2^{-j} |\topp (T)| \}.
\eeqa
The reason for this is that since the tiles $s\in T_P ^-$ are shorter than $P$, their frequency intervals can be much larger than that of $P$, 
meaning we lose control on $|P \cap \supp (L_P^{-})|$.  We use the extra decay from Schwartz tails to compensate for this.  The upper bound of $\size (T) \leq \sigma$ implies that for individual tiles $s\in T$ we have
$|\langle f , \varphi _s \rangle | \leq \sigma |s| ^{1\over 2}$.  Hence
\beqa
|\sum _{s\in T_j} \langle f, \varphi _s \rangle \phi _s \one _E |
& \lesssim & \sum _{s\in T_j} \sigma \chi _s ^{(\infty)} \\ 
& \lesssim & \sigma  2^{-Nk} \sum _{m \geq 2^{j+l} } m^{-N}  \\
& \lesssim & \sigma 2^{-Nk} 2^{ - Nj \over 2} 2^{-Nl \over 2}.
\eeqa
But note that since $\dense (s) \lesssim \delta$, we have
\beqa
\delta & \gtrsim & \int _{E_s } \chi _s ^{(1)} \\
& \geq & 2^{-100 (k+j+l)} {|P \cap \supp (\sum _{s\in T_j} \langle f, \varphi _s \rangle \phi _s \one _E )| \over |P|} , 
\eeqa
This last estimate follows from considering the distance between $s$ and $P$ relative to the length of $s$.
Hence for any $P \in \calp _{k,l}$, we have
\beqa
\int _P |L_P ^{-}| &\leq& \sum _{j \geq 0} \int _P |\sum _{s\in T_j} \langle f, \varphi _s \rangle \phi _s \one _E )| \\
& \lesssim & \sigma \sum _{j \geq 0} 2^{-Nk} 2^{ - Nj \over 2} 2^{-Nl \over 2} 
	|P \cap \supp (\sum _{s\in T_j} \langle f, \varphi _s \rangle \phi _s \one _E | \\
& \lesssim & \delta |P| \sigma 2^{-10(l+k)} 
\eeqa


Summing over $k,l$ and $P$ gives us
\beqa
 \sum _{P \in \calp }  \int_P | L_P ^- |
& \lesssim & \sum _{l \geq 0} \sum _{k\geq 0} \sum _{P\in \calp _{k,l} }  \int_P | L_P ^- |  \\
& \lesssim & \sum _{l \geq 0} \sum _{k\geq 0} \sum _{P\in \calp _{k,l} } \sigma \delta |P| 2^{-10k} 2^{-10l} \\
& \lesssim & \sigma \delta |\topp (T)|,
\eeqa
with the primary contribution coming from $P$ near $\topp (T)$ as usual.


\subsection{Large tiles}
We start by remarking that sorting with respect to horizontal distance from $T$ (i.e., using the index $l$, as in the previous subsection) is unnecessary in this subsection.  For if $P\in \calp _{k,l}$ with $l\geq C$, then $T_P ^+$ is empty, because 
$|\Pi _1 (P)| > | \Pi _1 (\topp (T)) |$.  This fact will appear several times in what follows.
Next, we show that the term under consideration in this section has small support.  Precisely:

\begin{claim} \label{support}
For $P \in \calp _k$, $L_P ^+ \one _E $ is supported on a set of size $\lesssim \delta |P| 2^{100k}$.
\end{claim}
The factor $2^{100k}$ arises from the tail in the definition of $\dense$ and the fact that $P$ is away from $\topp (T)$.  Fortunately, the decay in the functions $\varphi _s$ for $s\in T$ is even greater when $P$ is away from $\topp (T)$.
\begin{proof}
It is convenient to proceed by contradiction.  Assume $L_P ^+ \one _E $ has much larger support than
$\delta |P| 2^{100k}$.  By the construction of $P$, we know that there is some $s\in T$ such that $s\subseteq C2^k P$.  But this implies there is $R$ of the same dimensions as $P$, but located spatially over $T$, with $\omega _{R} \subseteq \omega _s$ and such that $\dense (R) \geq 100 \delta $, say.  Since this implies $s\leq R$, we have contradicted the assumption that $\udense (s) \leq \delta $.
\end{proof}

We now turn our attention to the second term in \eqref{star}.  Recall the definitions of $1$-trees and $2$-trees.  Clearly for every $s\in T$, either $\omega_{s,1} \cap \omega _{\topp (T)} = \emptyset$ or $\omega_{s,2} \cap \omega _{\topp (T)} = \emptyset$, so our tree $T$ can be partitioned as $T= T_1 \cup T_2$, where $T_j$ is a $j$-tree.  Let
\beqa
(T_P ^ + ) _j = T_P ^+ \cap T_j
\eeqa
for $j=1,2$.  Of course $(T_P ^ + ) _j$ is still a $j$-tree.  We treat the two cases separately.


\subsubsection{ The $2$-tree case }
This case is a bit easier to handle because of the location of the support of the function $\phi _s$.  More to the point:  Since $T_2$ is a $2$-tree, if there exists $x$ such that $\phi _s (x) \phi _t (x) \neq 0$ for $s, t \in T_2$, then $|s|= |t|$.  This follows from the fact that $\phi _s (x)=0$ unless $v(x) \in \omega _{s,2}$, together with the fact that $\omega _{s,1} \supseteq \omega _{\topp (T)}$, and similarly for $t$.  (This was mentioned near the definition of $\phi _s$ in Section \ref{modelreduction}.)  Further, we know that for any tile $s \in T$, we have $|\langle f , \varphi _s \rangle | \leq \sigma |s|^{1\over 2}$ by the $\size$ estimate for $T$.   Combining these observations with Claim \ref{support} and the rapid decay of $\phi _s$ in the vertical direction gives us for
$P\in \calp _k$ that
\beqa
\int _P \sum _{s\in (T_P ^ + )_2 } \langle f, \varphi _s \rangle \phi _s \one _E \lesssim \sigma \delta 2^{-10k} |P|,
\eeqa
since the integrand is uniformly bounded by $\sigma 2^{-200k}$.  As mentioned earlier, if $|\pi _1 (s)| \geq | \pi _1 (P)|$, then $\pi _1 (P) \subseteq C \pi _1 (\topp (T))$.  Hence
\beqa
\sum _k \sum _{P \in \calp _k} \int _P \sum _{s\in (T_P ^ + )_2 } \langle f, \varphi _s \rangle \phi _s \one _E
\lesssim \delta \sigma | \topp (T) |.
\eeqa
This completes the estimate for $T_2$.


\subsubsection{ The $1$-tree case } \label{treeproof}

In this case we appeal to orthogonality in the form of the Bessel inequality in Lemma \ref{bessel}.  For parallelograms $P\in\calp$ whose vertical component is large, we need the decay factor from Lemma \ref{bessel}.  We first introduce some extra functions associated to the tiles:  let
\beqa
\alpha _s (x) = \int \psi _s (t) \varphi _s (x_1 - t, x_2  ) dt.
\eeqa
The difference between $\alpha _s$ and $\phi _s$ is that the vector field $v$ makes no explicit appearance in the definition of $\alpha _s$; rather, the integral is taken over a horizontal line for every $x$.  In $\phi _s$, however, the integral is taken over an 
\it almost \rm horizontal line, where the precise definition of \it almost \rm depends on the length of $s$.  (The line is horizontal because we assumed that the slope of the long side of $\topp (T)$ is zero.  In the general case it is parallel to $\topp (T)$.)  
 We have the obvious equality
\beqa
\int _P \sum _{s\in (T_P ^ + )_1 }\epsilon _s  \langle f, \varphi _s \rangle \phi _s \one _E
 &=&
\int _P \sum _{s\in (T_P ^ + )_1 } \epsilon _s \langle f, \varphi _s \rangle \alpha _s \one _E \\
&+&
\int _P \sum _{s\in (T_P ^ + )_1 } \epsilon _s \langle f, \varphi _s \rangle (\phi _s - \alpha _s ) \one _E .
\eeqa
This decomposition allows us to reduce our problem to proving the following two claims:
\begin{claim} \label{average}
For each $P\in \calp$,
\beqa
\int _P \sum _{s\in (T_P ^ + )_1 } \epsilon _s \langle f, \varphi _s \rangle \alpha _s \one _E
\lesssim \delta \sum _{j \geq 0} 2^{-Nj} {1\over |2^j P|} \int _{2^j P} | \sum _{s\in T_1 } \langle f, \varphi _s \rangle \alpha _s |.
\eeqa
\end{claim}
\begin{claim} \label{error}
For $P \in \calp _k$,
\beqa
 \sum _{s\in (T_P ^ + )_1 } \epsilon _s \langle f, \varphi _s \rangle (\phi _s - \alpha _s ) \one _E \lesssim 2^{-200k} \sigma.
\eeqa
\end{claim}
Notice that $\supp \widehat{\alpha _s} \subseteq \supp \widehat{\varphi _s}$, since 
\beqa
\widehat{ \alpha _s} (\xi ) = \int \psi _s (t) e^{-2\pi i t \xi _1 } \widehat{ \varphi _s } (\xi ) dt.
\eeqa
This will allow us to prove orthogonality statements about the $\alpha _s$ later in the proof.  For example, 
   From this we can conclude that
\beqan \label{alphabessel}
|| \sum _{s\in T_1}  \epsilon _s \langle f, \varphi _s \rangle \alpha _s   || _2 ^2
\lesssim \sum _{s\in T_1} |\langle f, \varphi _s \rangle | ^2 ,
\eeqan
because the fact stated above about the Fourier support of the functions $\alpha _s$ allows us to prove this inequality in the same way we prove the Bessel inequality in Section \ref{besselsection}:  expand the square, and notice that $\langle \alpha _s , \alpha _t \rangle =0$ unless $|s| = |t|$.

Again we remark that if $T_P ^+$ is nonempty, then $\pi _1 (P) \subseteq C \pi _1 (\topp T)$.  Hence in the summation below we can ignore dependence on the parameter $l$ used in the last section.  Given these claims, together with Claim \ref{support}, we control the first term in \eqref{star} by
\beqa
 \sum _{P \in \calp }  \int_P L_P ^ +  &\lesssim &
\sum _{P \in \calp } \int _P  \epsilon _s \sum _{s\in (T_P ^ + )_1 } \langle f, \varphi _s \rangle \alpha _s \one _E \\
&+& \sum _{P \in \calp } \int _P  \epsilon _s \sum _{s\in (T_P ^ + )_1 } \langle f, \varphi _s \rangle (\phi _s - \alpha _s ) \one _E \\
&\lesssim &
\sum _k \sum _{P \in \calp _{k} } \delta  \int _P  
\sum _{j \geq 0} 2^{-Nj} {1\over |2^j P|} \int _{2^j P} | \sum _{s\in T_1 } \langle f, \varphi _s \rangle \alpha _s | \\
&+& \sum _k \sum _{P \in \calp _{k} }   2^{-200k}\sigma |P \cap \supp (L_P^+ )|.
\eeqa
Note that the second term in the last display is controlled by Claim \ref{support}.
For $P\in \calp _k$, it is convenient to split the function 
$ \sum _{s\in T_1 } \langle f, \varphi _s \rangle \alpha _s$ into two pieces, using the identity 
$\one _{\bbr ^2} = \one _{D_{k-5}} + \one _{(D_{k-5})^c}$, where 
\beqa
D_k= \{ (x,y) \colon |y| \lesssim 2^k |\pi _2 (\topp (T))| \}.
\eeqa
In other words, $D_k$ is horizontal strip of width $\sim 2^k |\pi _2 (\topp (T))|$.
(Obvious modifications can be made in the case $k\leq 5$.)
For the first piece-- the one closer to $\topp (T)$-- we can use the fact that the tile $P$ is far from $\topp (T)$ together with the decay in $j$ to obtain good control.  For the second piece-- the one away from $\topp (T)$ -- we can take advantage of the decay in the wave packets associated to tiles in $T$ in the form of the Bessel inequality in Lemma \ref{bessel}.
We focus first on the term close to $\topp (T)$:
\beqa
&\left. \right. & \sum _k  \sum _{P \in \calp _{k} } \delta  \int _P  
\sum _{j \geq 0} 2^{-Nj} {1\over |2^j P|} \int _{2^j P} 
	| \sum _{s\in T_1 } \langle f, \varphi _s \rangle \alpha _s  \one _{D_{k-5}} | \\
& = & \sum _k \sum _{P \in \calp _{k} } \delta  \int _P  
\sum _{j \geq k} 2^{-Nj} {1\over |2^j P|} \int _{2^j P} 
	| \sum _{s\in T_1 } \langle f, \varphi _s \rangle \alpha _s  \one _{D_{k-5}}| \\ 
& \lesssim &  \sum_k \sum_{P\in \calp _k} \delta 2^{-Nk} \int _P M ( | \sum _{s\in T_1 } \langle f, \varphi _s \rangle \alpha _s  \one _{D_{k-5}}|  ) \\
& = & \delta \int _{\cup _{l=0} ^C \cup _{P \in \calp _{k,l} } P } 
		2^{-Nk} M ( | \sum _{s\in T_1 } \langle f, \varphi _s \rangle \alpha _s  \one _{D_{k-5}}|  ) \\
& \lesssim & \delta 2^{-Nk} \left| \cup _{l=0} ^C \cup _{P \in \calp _{k,l} } P \right| ^{1\over 2}
		\left( \int | \sum _{s\in T_1 } \langle f, \varphi _s \rangle \alpha _s  \one _{D_{k-5}}|^ 2 \right) ^{1\over 2}.
\eeqa
This nearly finishes the proof for the first term, since we may estimate this $L^2$ norm by using orthogonality in the $x$-variable just as in the proof of Lemma \ref{bessel} below.  (Readers uncomfortable with this should look to the proof of Lemma \ref{bessel}.)  Specifically, we have
\beqa
\int | \sum _{s\in T_1 } \langle f, \varphi _s \rangle \alpha _s  \one _{D_{k-5}}|^ 2  
&=& \sum _{s\in T_1} \sum _{s' \in T_1} \langle f, \varphi _s \rangle \langle f, \varphi _{s'} \rangle \int _{D_k} \alpha _s \alpha _{s'} \\
& \lesssim & \sum _{s\in T_1} |\langle \varphi _s , f \rangle |^2 \sum _{s' \colon |s|=|s'|} \int |\alpha _s \alpha _{s'}| \\ 
& \lesssim & \sum _{s\in T_1} |\langle \varphi _s , f \rangle |^2 \\
& \lesssim & \sigma ^2 |\topp (T)|.
\eeqa
We have used symmetry and the $x$-orthogonality in the first inequality above.  This finishes the proof for the first term.  To control the second term (the one away from $\topp (T)$), we can appeal directly to a Bessel-type inequality.  Here we use such an inequality for the functions $\alpha _s$ rather than the functions $\varphi _s$, just as in the estimate above, but we also obtain significant decay in $k$ just as in Lemma \ref{bessel}.  The proof is identical to the proof of Lemma \ref{bessel}.  Hence
\beqa
&\left. \right. & \sum _k \sum _{l=0} ^ C \sum _{P \in \calp _{k,l} } \delta  \int _P  
\sum _{j \geq 0} 2^{-Nj} {1\over |2^j P|} \int _{2^j P} 
	| \sum _{s\in T_1 } \langle f, \varphi _s \rangle \alpha _s  \one _{(D_{k-5})^c} | \\
&\lesssim & \delta  \int _{\cup _{l=0} ^C \cup _{P \in \calp _{k,l} } P } 
	M (| \sum _{s\in T_1 } \langle f, \varphi _s \rangle \alpha _s  \one _{(D_{k-5})^c} |) \\
&\lesssim & \delta  \left| \cup _{l=0} ^C \cup _{P \in \calp _{k,l} } P \right| ^{1\over 2}
	\left( \int  | \sum _{s\in T_1 } \langle f, \varphi _s \rangle \alpha _s  \one _{(D_{k-5})^c} |^2  \right) ^{1\over 2} \\
&\lesssim & \delta  2^k |\topp (T)|^{1\over 2} (\sigma ^2 2^{-100k} |\topp (T)| ) ^{1\over 2} \\
&\lesssim & 2^{-10k} \delta  \sigma |\topp (T)|,	
\eeqa
which is what we want.

\begin{proof} [Proof of Claim \ref{average}]
Recall that we are considering a point $x\in P$ for some parallelogram $P$, and we consider the sum
\beqa
\sum _{s\in T_1 \colon |\pi _1 (s)| > |\pi _1 (P)| } \langle f, \varphi _s \rangle \phi _s (x).
\eeqa
The restriction in the summation already implies that for any $x$, there is $m(x)$ such that all tiles $s$ who make an appearance in the sum above satisfy $|\pi _1 (s)| \geq m(x)$.  Further, since we know that $u(x)\in \omega _{s,2}$, we also have $M(x)$ such that all tiles $s$ who make an appearance in the sum above satisfy $|\pi _1 (s)| \leq M(x)$.  Both of these claims are reversible, so 
\beqa
\{s\in T_1 \colon |\pi _1 (s)| > |\pi _1 (P)| \} = \{s\in T \colon m(x) \leq L(s) \leq M(x) \}.
\eeqa
Hence it is our goal to estimate
\beqa
\sum _{s\in T \colon m(x) \leq L(s) \leq M(x) } \langle f, \varphi _s \rangle  \alpha _s .
\eeqa
%
Denote by $k$ a Schwartz function such that 
$\supp \hat{k} \subseteq [-1-{1\over 100}, 1+ {1\over 100}]^2$, and such that $\hat{k} (\xi) = 1$ for $\xi \in [-1 , 1] ^2$. further denote by $k_r$ the function obtained by adapting $k$ to the rectangle $[{-1\over r}, {1\over r}] \times [{-1\over w}, {1 \over w}]$; i.e., let $k_r (x,y) = k({x \over r}, {y \over w})$.   With this definition, we know for any $N$ (which appears in the last line of the computation below)
\beqa
\sum _{s\in T_1 \colon m(x) \leq L(s) \leq M(x) } \langle f, \varphi _s \rangle   \alpha _s
&=& \sum _{s\in T_1 \colon m(x) \leq L(s)  } \langle f, \varphi _s \rangle   \alpha _s \\
&-& \sum _{s\in T_1 \colon L(s) > M(x) } \langle f, \varphi _s \rangle   \alpha _s \\
&=& (\sum _{s\in T_1  } \langle f, \varphi _s \rangle   \alpha _s   ) \ast k_{m(x)} \\
&-& (\sum _{s\in T_1  } \langle f, \varphi _s \rangle   \alpha _s   ) \ast k_{M(x)} \\
&\leq & \sum _{j \geq 0} 2^{-Nj} {1\over |2^j P|} \int _{2^j P} | \sum _{s\in T_1 } \langle f, \varphi _s \rangle \alpha _s |.
\eeqa

\end{proof}

\begin{proof} [Proof of Claim \ref{error}]
By the argument at the beginning of the proof of Claim \ref{average}, it suffices to estimate
\beqa
\sum _{s\in T \colon m(x) \leq |\pi _1(s)| \leq M(x) } \langle f, \varphi _s \rangle (\phi _s (x)- \alpha _s (x) )\one _{\omega_{s,2}} (u(x)).
\eeqa
To do this we first estimate $|\phi _s - \alpha _s|$.  By definition, we have
\beqa
|\phi _s (x)- \alpha _s (x) | &\leq &
\int |\psi _s (t)| | \varphi _s (x_1 -t, x_2 - t u(x) ) - \varphi (x_1 - t , x_2 ) | dt.
\eeqa
To compute the difference in the integrand, estimating the following quantity will be helpful:
\beqa
\star : = \sup _{z\in [0, t u(x)]} {\partial \over \partial x_2 } \varphi _s (x_1-t, x_2 - z) .
\eeqa
Fix an integer $j\geq 1$ and consider $|t| \sim 2^j |\pi _1 (s)|$.  If $(x_1, x_2) \not\in 2^{j+10} s$, then 
$\star \lesssim \chi _s ^{(2)} (x_1, x_2)$.  If $(x_1, x_2) \in 2^{j+10} s$, then $\star \lesssim 1$.  We also have that 
$\psi _s (t) \lesssim {1\over 2^{Nj} |s|}$ for any $N$.  Analogous facts hold when $j=0$ and $|t| \leq |\pi _1 (s)|$.  
Let $I_j = \{t \colon |t| \sim 2^j |\pi _1 (s)| \} $ for $j\geq 1$ and 
$I_0 = \{ t \colon |t| \leq |\pi _1 (s)| \}$. Combining these observations gives us for $(x_1, x_2) \not\in 2^{j+10} s$ that 
\beqa
|\phi _s (x)- \alpha _s (x) | &\lesssim & \sum _{j\geq 0} \int _{I_j } 
			{1\over 2^{Nj} |s|} 2^j |\pi _1 (s)| {|u (x)| \over w} \chi _s ^{(2)} (x_1, x_2)  dt \\ 
		&\lesssim & |\pi _1 (s)| {|u (x)| \over w} \chi _s ^{(2)} (x_1, x_2).
\eeqa
If $(x_1, x_2) \in 2^{j+10} s$, then we have $\star \lesssim 2^{100j} \chi _s ^{(2)}$, so 
\beqa
|\phi _s (x)- \alpha _s (x) | &\lesssim & \sum _{j\geq 0} \int _{I_j } 
			{1\over 2^{-Nj} |s|} 2^j |\pi _1 (s)| {|u (x)| \over w}  dt \\ 
		& \lesssim & |\pi _1 (s)| {|u (x)| \over w} \chi _s ^{(2)} (x_1, x_2).
\eeqa
Since $u(x) \in \omega _{s,2}$ for all $s \in T_1$, we know $u(x) \leq {w \over {|\pi _1 (s)|}} $.  Combining this with the fact that
$|\langle f, \varphi _s \rangle | \lesssim \sigma |s|^{1\over 2}$ and the estimate immediately above, we have
\beqa
|\sum _{ m(x) \leq |\pi _1 (s)| \leq M(x) } \langle f, \varphi _s \rangle ( \phi _s - \alpha _s )|
&\leq &
\sum _{ |\pi _1 (s)| \leq {{w} \over {u(x)}}   }
		\sigma |s|^{1\over 2}  |u(x)| {{|\pi _1(s)|} \over {w}} \chi _s ^{(2)} (x_1, x_2) \\
&\lesssim & \sigma \chi ^{(\infty) } _ {\topp (T)} (x_1, x_2) ,
\eeqa
which is what we claimed.

\end{proof}


\section{Proof of $\size$ estimate } \label{sizesection}
%

In this section we write $f=\one _F$; note that we do not use the fact that $f$ is a characteristic function.
As with the tree lemma, there are small modifications required from the one-dimensional situation to handle Schwartz tails in the vertical direction.  We use the Bessel inequality from Lemma \ref{bessel} to do this.
First we note that by assumption,
\beqa
\sigma ^2 \sum_{T\in \calt} |\topp (T)|  &\lesssim &
 \sum_{T\in \calt} \sum _{s\in T} |\langle f, \varphi _s \rangle | ^2
 \\
&=&  \int f \sum _T \sum _s \langle f, \varphi _s \rangle \varphi _s \\
&\leq & ||f|| _2  || \sum _T \sum _s \langle f, \varphi _s \rangle \varphi _s || _2 .
\eeqa
It is enough to prove
\beqa
|| \sum _T \sum _s \langle f, \varphi _s \rangle \varphi _s || _2
\leq \sigma \sqrt{ \sum_{T\in \calt} |\topp (T)|  } .
\eeqa
By expanding the square and using symmetry, we have
\beqa
|| \sum _{T\in \calt} \sum _{s\in T} \langle f, \varphi _s \rangle \varphi _s || _2  ^2 &=&
\sum_{T\in \calt} \sum _{T' \in \calt} \sum _{s\in T'}  \sum _{s'\in T'}
	\langle f, \varphi _s \rangle \langle f, \varphi _{s'} \rangle \langle \varphi _s , \varphi _{s'} \rangle \\
& \lesssim & \sum _{T\in \calt} \sum _{s\in T} \sum _{T' \in \calt} \sum _{s'\in T' \colon |s' | = |s| }
	|\langle f, \varphi _s \rangle \langle f, \varphi _{s'} \rangle \langle \varphi _s , \varphi _{s'} \rangle| \\
&+& \left| \sum _{T\in \calt} \sum _{s\in T} \sum _{T' \in \calt} \sum _{s'\in T' \colon |s' | < |s| }
	\langle f, \varphi _s \rangle \langle f, \varphi _{s'} \rangle \langle \varphi _s , \varphi _{s'} \rangle   \right| \\
&=& B+C.
\eeqa
Note that
\beqa
\{ s' \colon |s'| = |s| \text{ and } \omega _s \cap \omega _{s'} \neq \emptyset \}
\eeqa
partitions $\bbr ^2$, so
\beqa
\sum _{|s'| =|s| } |\langle \varphi _s , \varphi _{s'} \rangle | \sim 1.
\eeqa
Hence we can estimate the first term, using symmetry again, by
\beqa
B &\lesssim & \sum _{T\in \calt} \sum _{s\in T} \sum _{T' \in \calt} \sum _{s'\in T \colon |s' | = |s| }
	|\langle f, \varphi _s \rangle| ^2 |\langle \varphi _s , \varphi _{s'} \rangle| \\
&\lesssim &  \sum _{T\in \calt} \sum _{s\in T} |\langle f, \varphi _s \rangle| ^2 \\
&\sim & \sigma ^2 \sum_{T\in \calt} |\topp (T)|.
\eeqa
Now we look at the second term $C$.  By Cauchy-Schwarz, we have
\beqa
C &\leq & \sum _{T\in \calt} \left( \sum _{s\in T}|\langle f, \varphi _s \rangle| ^2   \right) ^{1\over 2}
	\left( \sum _{s\in T} \left| \sum _{T' \in \calt ' } \sum _{s'\in T' \colon |s' | < |s| }
		\langle \varphi _s , \varphi _{s'} \rangle \langle f, \varphi _{s'} \rangle \right|^2 \right)  ^{1\over 2} \\
  &\lesssim & \sum _{T\in \calt} \sigma  |\topp (T)| ^{1\over 2}  D(T)^{1\over 2}
\eeqa
where
\beqa
D(T)= \sum _{s\in T} \left| \sum _{T' \in \calt ' } \sum _{s'\in T' \colon |s' | < |s| }
		\langle \varphi _s , \varphi _{s'} \rangle \langle f, \varphi _{s'} \rangle \right|^2 .
\eeqa
It remains to analyze $D(T)$ for a tree $T\in \calt$.  We claim that the set of tiles over which the inner sum ranges is actually independent of $s$.  More specifically, define
\beqa
\cala = \{ s' \in \bigcup _{T' \neq T, T' \in \calt } T' \colon \omega _{s, 1} \cap \omega_{s', 1} \neq \emptyset
		\text{  and  } |s'| < |s| \text{  for some   } s\in \calt \}.
\eeqa
Then 
\begin{claim}
For each $s\in T$,
\beqa
\sum_{T' \in \calt} \sum _{s'\in T' \colon |s' | < |s| }
		\langle \varphi _s , \varphi _{s'} \rangle \langle f, \varphi _{s'} \rangle
= \sum _{s'\in \cala}
		\langle \varphi _s , \varphi _{s'} \rangle \langle f, \varphi _{s'} \rangle .
\eeqa
\end{claim}
\begin{proof}
It is clear from the definition of $\cala$ that the summation on the left is over a set of tiles that is contained in $\cala$.  So suppose $s' \in \cala$; by definition of $\cala$, this gives us $\tilde{s} \in \calt$ such that $|s'| < |\tilde{s}|$ and such that 
$ \omega _{\tilde{s}, 1} \cap \omega_{s', 1} \neq \emptyset$.  This last condition guarantees that $\omega_{s', 1} \supseteq \omega _T$.  If $|s| \geq |\tilde{s}|$, then of course $|s| > |s'|$ and $ \omega _{s, 1} \cap \omega_{s', 1} \neq \emptyset$, so that in fact the tile $s'$ appears in the summation on the left hand side of the claim.  If $|s| < |\tilde{s}|$ and $|s| > |s'|$ then we are done as before.  So assume $|s| \leq |s'| < |\tilde{s}|$.  In this case $\omega _{s, 1} \cap \omega _{s', 1} = \emptyset$, which implies that 
$\langle \varphi _s , \varphi _{s'} \rangle =0$, finishing the proof of the claim.
\end{proof}
Now for a collection of tiles $\calc$, define
\beqa
F(\calc) = \sum _{t \in \calc} \langle f, \varphi _{t} \rangle \varphi _{t}.
\eeqa
With this notation, we have
\beqa
D(T) = \sum _{s\in T} |\langle \varphi _s, F(\cala)\rangle  |^2 .
\eeqa
Before we proceed, we mention a key disjointness property of tiles in $\cala$.  

\begin{claim} \label{technicalsizeclaim}
Tiles in $\cala$ are pairwise disjoint.
\end{claim}
\begin{proof}
Suppose $t, t' \in \cala$.  Then there are $s, s' \in \calt$ such that $\omega _{t, 2} \supseteq \omega _s \supseteq \omega _{\topp(T)}$ and such that $\omega _{t', 2} \supseteq \omega _{s'} \supseteq \omega _{\topp (T')} $.  Hence $\omega _{t, 2} \cap \omega _{t', 2} \neq \emptyset$, we may assume without loss of generality that $\omega _{t, 2} \subseteq \omega _{t', 2}$, i.e., that $|t'| \leq |t|$.  This means the tree $T^*$ containing $t$ was selected before the tree containing $t'$.  Finally, note that $t$ and $t'$ cannot belong to the same $1$-tree, since 
$\omega _{t, 2} \subseteq \omega _{t', 2}$.  If $t\cap t' \neq \emptyset$, then in fact $t' \subseteq V (\topp (T^*))$, and hence $t'$ was included in the maximal tree $\widetilde{T^*}$ containing the $1$-tree $T^*$; see the selection algorithm in Section \ref{organizationsection} for construction of this tree $\widetilde{T^*}$.  Hence the tiles in $\cala$ are pairwise disjoint.  
\end{proof}
We now introduce some more notation to sort the tiles in $\cala $ according to how far they are from
$\topp (T)$.  For $k>1$, let $R_k = 2^k \topp (T)$.  Let $R_0 = \topp (T)$.  Then let
\beqa
\cala _k = \{ s' \in \cala \colon s' \subseteq R_k \text{  but  } s' \not\subseteq R_{k-1} \}.
\eeqa
Now by Minkowski,
\beqa
\left(\sum _{s\in T} |\langle \varphi _s, F(\cala) \rangle|^2  \right)^{1\over 2}
\leq \sum _k \left(\sum _{s\in T} |\langle \varphi _s, F(\cala _k) \rangle | ^2 \right) ^{1\over 2}.
\eeqa
It remains to show
\beqan \label{sizedecay}
\sum _{s\in T} |\langle \varphi _s, F(\cala _k) \rangle | ^2 \lesssim  2^{-10k} \sigma |\topp (T)|.
\eeqan
We will use the spatial localization of the tiles $s\in T$ to $\topp (T)$ to obtain the desired decay in $k$.
We have
\beqa
\sum _{s\in T} |\langle \varphi _s, F(\cala _k) \rangle | ^2
& \lesssim & \sum _{s\in T} |\langle \varphi _s,\one _{R_{k-3} } F(\cala _k) \rangle | ^2
	  +  \sum _{s\in T} |\langle \varphi _s, \one _{R_{k-3}^c } F(\cala _k) \rangle | ^2 \\
&=&  I _k + II _k .
\eeqa
First we estimate $I_k$.  For $x\in R_{k-3}$ and $s\in \cala _k$, we have 
\beqa
\varphi _s (x) \one_{R_{k-3}} (x) \lesssim 2^{-10k} {1\over {\sqrt{|s|}} } \chi _s ^{(\infty)} (x).  
\eeqa
We now estimate $||\one _{R_{k-3}} F(\cala _k)||_2$ by duality.  We make one small observation as a preliminary:  
\begin{claim}
If $M$ is the strong maximal operator, then 
\beqa
\int \chi _s ^{(\infty)} (x) g(x)dx \lesssim \int _s Mg(x) dx.
\eeqa
\end{claim}
We remark that each $s\in \cala$ is essentially pointed in the direction of $T$, so the strong maximal operator is appropriate here.
\begin{proof}
\beqa
\int \chi _s ^{(\infty)} (x) g(x)dx & \lesssim & |s| \sum _{k\geq 0} 2^{-3k} {1\over {|2^k s|}} \int _{|2^k s|} |g| \\
&\lesssim & |s| \inf _{x\in s} M g(x) \\
&\lesssim &  \int _s Mg(x) dx.
\eeqa
\end{proof}
Consider a function $g \in L^2$, and remember that 
$ | \langle f , \varphi _s \rangle | \lesssim \sigma \sqrt{|s|}$.  Then using the claim above about disjointness of tiles $s\in \cala _k$, we have
\beqa
\int F(\cala _k) g \one_{R_{k-3}} &=&
\int \sum _{s \in \cala _k} \langle f , \varphi _s \rangle \varphi _s (x) \one_{R_{k-3}} (x) g \\
&\lesssim & \int \sum _{s \in \cala _k} 2^{-10k} \sigma \chi _s ^{(\infty)} (x) g \\
&\lesssim & 2^{-10k} \sigma \sum _{s \in \cala _k} \int _s M g \\
& \leq & 2^{-10k} \sigma  \int_{\bigcup _{s \in \cala _k} s} M g \\
& \leq & 2^{-10k} \sigma \left| R_k \right| ^{1\over 2} ||g||_2 \\
& \leq & 2^{-10k} \sigma  (2^{2k} |\topp (T)|) ^{1\over 2}  ||g||_2,
\eeqa
which implies that 
\beqa
I_k \lesssim || \one _{R_{k-3}} F( \cala _k) || _2 ^2 \lesssim \sigma ^2 2^{-4k} |\topp (T)|.
\eeqa
This proves \eqref{sizedecay} for $I_k$.

To estimate $II_k$, we need only estimate $|| F( \cala _k) || _2$ and apply Lemma \ref{bessel}.  We do this just as above:  let $g$ be such that $||g|| _2 = 1$.  Then 
\beqa
 \int F (\cala _k) g &\leq& \int | \sum _{s \in \cala _k} \langle f , \varphi _s \rangle \varphi _s g |\\
&\lesssim & \int |\sum _{s \in \cala _k} \sigma \chi _s ^{(\infty)} g | \\
& \lesssim & \sigma \int _{\cup _{s\in \cala _k} s } Mg \\
& \lesssim & \sigma |R_k | ^{1\over 2} .
\eeqa
So 
\beqa
||F(\cala_k) || _2 ^2 \lesssim \sigma ^2 |\cup \cala _k | \lesssim \sigma ^2 2^{2k} |\topp (T)|.
\eeqa
Hence by Lemma \ref{bessel},
\beqa
II_k \lesssim 2^{-10k} ||F(\cala_k) || _2 ^2 \lesssim \sigma ^2 2^{-8k} |\topp (T)|.
\eeqa
Summing in $k$ proves $D(T) \lesssim \sigma ^2 |\topp (T)|$, which finishes the proof.


\section{Localized Bessel inequality} \label{besselsection}

In this section we prove a Bessel inequality for $1$-trees with functions supported away from the top of the tree.  Specifically:
\begin{lemma} \label{bessel}
Let $T$ be a $1$-tree.  For $k\geq 1$, let $R_k = 2^k \topp (T)$.
For $k\geq 1$, let $\Omega _k = R_{k} \setminus R_{k-1}$.  Define $\Omega _0 = \topp (T)$.
Then for any $N>0$,
\beqa
\sum_{s\in T} |\langle f \one _{\Omega _k} , \varphi _s \rangle | ^2
	\lesssim 2^{-Nk} ||f \one _{\Omega _k}||_2 ^2.
\eeqa
\end{lemma}
\begin{rem}
For a classical $1$-dimensional tree, this can be proved by using the extreme spatial decay of the wave packets $\varphi _s $, $s\in T$, away from $\topp (T)$.  We use this in conjunction with orthogonality in the $x$-variable to  handle interactions of functions $\varphi _s$, $\varphi _{s'}$ horizontally close to the tree, where tail estimates do not improve for shorter tiles in the tree.  This is the reason for the decomposition of $\Omega _k$ into $\calb _k$ and $\calc _k$ in the proof below.  
\end{rem}
\begin{proof}
For notational convenience, we will assume that the parallelogram $\topp(T)$ is centered at the origin, has width $1$, and has sides parallel to the coordinate axes.
First note that
\beqa
\sqrt{ \sum_{s\in T} |\langle f \one _{\Omega _k} , \varphi _s \rangle | ^2 } 
&=& 	\sqrt{ \sum_{s\in T} |\langle f \one _{\bold B _k} , \varphi _s \rangle | ^2 } 
	\sqrt{ \sum_{s\in T} |\langle f \one _{\bold C _k} , \varphi _s \rangle | ^2 } \\
&=:& B + C,
\eeqa
where
\beqa
\calb _k &=& \{ (x,y) \in \Omega _k \colon |y| \geq 2^k \} \\
\calc _k &=& \Omega _k \setminus \calb _k .
\eeqa
To estimate $B$ we will need to use orthogonality in the horizontal variable.  To estimate $C$ we will need only spatial decay, as in the one-dimensional case.

Note that by Cauchy-Schwarz
\beqa
B^2 &=& \int _{\mathbf B _k} f \sum _{s\in T} \langle f \one _{\mathbf B _k} , \varphi _s \rangle \varphi _s \\
& \leq & ||f \one _{\Omega _k}||_2 \left( \sum _{s\in T} \sum _{s\in T'} \int _{|y|\geq 2^k} \int _{x\in \bbr}
	\langle f \one _{\mathbf B _k} , \varphi _s \rangle  \langle f \one _{\mathbf B _k} , \varphi _{s'} \rangle
			\varphi _s (x,y) \varphi _{s'} (x,y)	dx dy   \right) ^{1\over 2} .
\eeqa
Also note that if $|s| \neq |s'|$, then for every $y$, we have
\beqa
\int _x \varphi _s (x,y) \varphi _{s'} (x,y) =0.
\eeqa
This follows from the definition of the wave packets $\varphi _s$; specifically, note that $\pi _1 (\supp ( \hat{\varphi} _s ) ) \cap \pi _1 (\supp ( \hat{\varphi} _{s'} ) ) = \emptyset $ whenever $\omega _{s,1} \cap \omega _{s,2} = \emptyset$, which happens whenever $s, s'$ are in the same $1$-tree and $|s| \neq |s'|$.
By symmetry we may estimate $|\langle f \one _{\Omega _k},  \varphi _s \rangle  \langle f \one _{\Omega _k} , \varphi _{s'} \rangle |\leq |\langle f \one _{\Omega _k},  \varphi _s \rangle|^2 $, which gives us

\beqa
&&\sum _{s\in T} \sum _{s'\in T} \int  _{|y|\geq 2^k}  \int _x
	\langle f \one _{B _k} , \varphi _s \rangle  \langle f \one _{\Omega _k} , \varphi _{s'} \rangle
			\varphi _s (x,y) \varphi _{s'} (x,y) \\
&\leq& \sum _{s\in T} \sum _{s'\in T \colon |s|=|s'|}
	|\langle f \one _{B _k}  \varphi _s \rangle|^2
	\int _{|y|\geq 2^k} \int_x   |\varphi _s || \varphi _{s'} |.
\eeqa
But note that
\beqa
\sum _{s'\in T \colon |s|=|s'|} \int _{|y|\geq 2^k} \int_x  |\varphi _s || \varphi _{s'} | \leq 2^{-Nk},
\eeqa
because the prototype $\varphi $ is Schwartz, $s\in T$, and $\Omega _k$ is far away from $\topp (T)$.
Hence
\beqa
B \lesssim 2^{-{N\over 2}k} ||f \one _{\Omega _k}||_2
	\left( \sum _{s\in T} |\langle f \one _{\Omega _k}  \varphi _s \rangle|^2  \right) ^{1\over 2}.
\eeqa
We now estimate $C$.  Define
\beqa
T^j = \{ s\in T \colon |s| = 2^{-j} |\topp (T)| \}.
\eeqa
Note that if $s\in T^j$, then
$|\langle f \one _{\mathbf{C} _k} , \varphi _s \rangle| \lesssim 2^{-{N\over 2}k -50j} ||f \one _{\Omega _k}||_2 $
by Cauchy-Schwarz and the fact that $||\varphi _s  \one _{ \mathbf{c} _k } ||_2 \lesssim 2^{-{N\over 2}k -50j}$.  This last claim follows from the fact that $\varphi _s$ is highly localized to $\topp (T)$, and because $\mathbf{C}_k$ is far away from $\topp (T)$ horizontally.  (Of course we could not make the same argument for $B$ because we can do no better than $||\varphi _s \one _{\calb _k} ||_2 \lesssim 2^{-Nk}$ for $s\in T^j$; i.e., there is no decay in the parameter $j$.)  This is already enough:
\beqa
C \leq \sum _{j\geq 0} \sum _{s\in T^j} |\langle f \one _{\Omega _k}  \varphi _s \rangle|^2
\lesssim 2^{-{N\over 2}k} ||f \one _{\Omega _k}||_2,
\eeqa
which finishes the proof of the lemma.
\end{proof}


\section{Square function estimates} \label{sqfcnsection}
In this section we prove Lemma \ref{lpsquarefcn}.  The proof is similar to the standard proof of $L^p$ boundedness for the analogous one-dimensional square function, with a few tweaks to handle the two-dimensionality.  For notational convenience we will assume, without loss of generality, that the tree $T$ has top that is axis parallel and centered at the origin.  Proving the lemma  with the spatial localization requires us to decompose $\Delta$ spatially as follows.  For $k\geq 1$, define the set $\Omega_ k= 2^k \topp (T) \setminus 2^{k-1} \topp (T)$.  For $k=0$, define $\Omega _k = \topp (T)$.  Now define
\beqa
\Delta _k (f) = \left( \sum_{s\in T} |\langle f , \one _{\Omega _k } \varphi _s \rangle| ^2 {{\one _{ s } } \over {|s|}}  \right)^{1\over 2} .
\eeqa
By Minkowski's inequality, we have
\beqa
\Delta f (x) &=& \left( \sum_{s\in T} |\langle f , \sum _k \one _{\Omega _k } \varphi _s \rangle| ^2 {{\one _{ s }} \over {|s|}}  \right)^{1\over 2} \\
&\leq & \sum _k \Delta _k f (x)
\eeqa
pointwise, so again by Minkowski's inequality we have
\beqa
|| \Delta f ||_ p  \leq \sum _k || \Delta _k f|| _p.
\eeqa
We will prove that for any $N$,
\beqan \label{deltadecay}
|| \Delta _k f|| _p \lesssim 2^{-Nk} ||\one _{\Omega _k} f||_p .
\eeqan
With this, we can use H\"{o}lder's inequality to see that for any $N$, we have
\beqa
||\Delta f||_p & \lesssim & \sum _k 2^{-Nk} ||\one _{\Omega _k} f||_p \\
&\lesssim &  \left( \sum _k 2^{- Nk} \int _{\Omega _k} |f| ^p  \right) ^{1\over p} \\
&\lesssim &  \left(    \int | \beta _{N,T} f | ^p     \right) ^{1\over p},
\eeqa
where $\beta _{N,T}$ is the function defined in the statement of Lemma \ref{lpsquarefcn},
which finishes the proof of Lemma \ref{lpsquarefcn}.  It remains to prove \eqref{deltadecay}.  Note that Lemma \ref{bessel} is exactly this when $p=2$.  By interpolation, it is enough to prove the following weak type estimate:
\beqa
|\{ \Delta _k f > \lambda \} | \lesssim 2^{2k} {{||f|| _1 } \over {\lambda ^1 } }.
\eeqa
By dividing the function $f$ into $\lesssim 2^{2k}$ pieces, we may assume the support of $f$ is contained in a translate of $\topp (T)$.  With this assumption, it is enough to prove for such $f$ that 
\beqa
|\{ \Delta _k f > \lambda \} | \lesssim {{||f|| _1 } \over {\lambda ^1 } }.
\eeqa
Our argument proceeds more or less by the usual path of Calderon-Zygmund decomposition.


Denote by $R_k$ the rectangle with same center and length as $R$ but $2^k$ times the height.  
 Let $\widetilde{\calb}$ be the collection of maximal rectangles of width $w$ taken from the collection such that 
\beqa
{1\over |R_k|} \int _{R_k} |f| > 2^{5k} \lambda,
\eeqa
and for each $R\in \widetilde{\calb}$, let $R' = \pi _1(R) \times \pi_2 (C \topp (T))$.  Then let $\calb =\{ R' \colon R \in \widetilde{\calb} \}$.  We can see already that 
$\sum _{R\in \widetilde{\calb} } |R|  \leq \sum _{R\in \calb} |R| \lesssim {||f||_1 \over \lambda}$.  
This follows from the weak (1,1) inequality for the Hardy-Littlewood maximal function, which holds for rectangles of fixed width: if we write, for $k\geq 0$,
\beqa
\widetilde{\calb} _k = \{ R \in \widetilde{\calb} \colon {1\over |R_k|} \int _{R_k} |f| > 2^{5k} \lambda  \}, 
\eeqa
then we have 
\beqa
\left. \right.
\sum _{R\in \widetilde{\calb} } |R|  \\
&\lesssim & \sum _{k \geq 0 }  2^k {||f||_1 \over 2^{5k} \lambda }  \\
&\lesssim & {||f||_1 \over \lambda }.
\eeqa

For each $(x,y)\in R$, let 
\beqa
b(x,y) = f(x,y) - {1\over |\pi _1 (R)|} \int _{\pi _1 ( R )} f(z,y) dz.
\eeqa
Note that by definition we have for each $y\in \pi _2 (\topp (T))$ that 
\beqa
\int _{\pi _1 (R)} b(x,y)dx =0.
\eeqa
We also have the following helpful fact:
\begin{claim}
For each $y\in \pi _2 (C\topp (T))$, we have 
\beqa
{1\over |\pi _1 (R)|} \int _{\pi _1 ( R )} f(z,y) dz \leq C \lambda.
\eeqa
\end{claim}
\begin{proof} [Proof of Claim]
Note that $\widehat{f}$ is supported in the annulus of width ${1\over w}$.  Let $k$ be a function such that $\widehat{k}(\xi)=1$ for 
$\xi \in [-4w,  4w]$.  Then 
\beqa
f(x,y) = \int f(x,w) k(y-w) dw, 
\eeqa
so 
\beqa
{1\over |\pi _1 (R)|} \int _{\pi _1 (R)} |f(z,y)| dz 
&=& {1\over |\pi _1 (R)|} \int _{\pi _1 (R)} |\int f(z,w) k(y-w) dw| dz \\
\eeqa
Because $k$ rapidly decays away from a rectangle of height $w$, if we denote by $R_k$ the rectangle with same center and length as $R$ but $2^k$ times the height, then
\beqa
& \left. \right. & {1\over |\pi _1 (R)|} \int _{\pi _1 (R)} |\int f(z,w) k(y-w) dw| dz \\
& \lesssim & {1\over |\pi _1 (R)| } \int _{\pi _1 (R)} \sum_k {1\over 2^k} \int _{2_k} ^{2^k} f(z,w) 2^{-10k} dw dz \\
&\leq & \lambda,
\eeqa
where the last inequality is by assumption on $R$. 
\end{proof}
With this claim, we define 
\beqa
g(x,y) = f(x,y) \text{  for  } (x,y) \not\in \bigcup _{R\in \calb} R
\eeqa
and
\beqa
g(x,y) = {1\over |\pi _1 (R)|} \int _{\pi _1 ( R )} f(z,y) dz \text{  for  } (x,y) \in R \in \calb.
\eeqa
Note that by the claim we have $g(x,y) \lesssim \lambda$ for $(x,y) \in R$.  Further, for almost every 
$(x,y) \not\in \bigcup _{R\in \calb} R$ such that $g(x,y)=f(x,y) >>\lambda$, there exists a horizontal line segment $L$ through $(x,y)$ such that 
${1\over |L|} \int _L f >>\lambda$, which implies there is a rectangle of width $w$ containing $(x,y)$ on which the average of $f$ is larger than $\lambda$, contradicting our assumption that $(x,y) \not\in \bigcup _{R\in \calb} R$.  Hence $g \lesssim \lambda $ almost everywhere.

To see the purpose of including the rectangles $5CR'$ in the exceptional set (rather than a small dilate of $R$ itself), consider a rectangle $R$ north of the tree $T$, and a mean zero function $h$ supported on $R$.  Analysis of $\int _{(5CR) ^c } \Delta h $ is a bit more complicated than in the one-dimensional case because the collection $\{ \varphi _s \} _{s\in T} $ has no orthogonality in the vertical direction.  However by excluding $R'$, we need only consider small tiles $s$ supported away from the vertical translate of $5CR$,  allowing us to take advantage of the spatial decay (in the horizontal variable) of the functions $\varphi _s $.  

With this modification, the proof now proceeds as expected:  Use the fact that $|g| \lesssim \lambda$, together with the $L^2$ estimate on $\Delta $ to see
\beqa
|\{ \Delta _k g > \lambda \}| \lesssim {{\int |g|^2 } \over {\lambda ^2} } \lesssim {{||f||_1} \over {\lambda } }.
\eeqa
Additionally, by the Chebyshev and triangle inequalities, together with sublinearity of $\Delta _k$, we have
\beqa
|\{ x\not\in E \colon \Delta _k (\sum_R b_R ) > \lambda \}| &\leq & {1\over {\lambda }} \sum _R \int _{(5CR')^c } |\Delta _k(b_R)|.
\eeqa
To finish the proof we show that for each $R \in \calb$, we have
\beqan \label{outside}
\int _{( 5CR')^c } |\Delta _k(b_R)| \lesssim \int |b_R|,
\eeqan
which will give us that
\beqa
|\{ x\not\in E \colon \Delta (\sum_R b_R ) > \lambda \}| \lesssim {1\over {\lambda }} \sum _R \int |b_R|
\lesssim \sum _R |R| \lesssim {{||f||_1 } \over {\lambda }}.
\eeqa

Once again, to prove \eqref{outside}, we essentially follow the one-dimensional argument, dealing with a few extra nuisances along the way.  A reader having trouble seeing through the technicalities should note that all of the computations below are essentially the same as in the one-dimensional case.  The problem is understanding why the present situation is essentially the same as the one-dimensional case.  More specifically, to prove \eqref{outside}, it is convenient to make a few simplifying (and valid) assumptions.  
For each parallelogram $s\in T$ define
\beqa
\tilde{s} = \pi _1 (s) \times C \pi _2 (\topp (T)) .
\eeqa
Since $s \subseteq \tilde{s}$, it is clear that if we define
\beqan \label{newdelta}
\tilde{\Delta _k } f = \left( \sum_{s\in T} |\langle f \one _{\Omega_k}, \varphi _s \rangle| ^2 {{\one _{\tilde{s} }  } \over {|s|}}  \right)^{1\over 2} ,
\eeqan
then $\Delta _k f \leq \tilde{\Delta _k} f$ pointwise.  For each $s\in T$, we know that $\pi _1 (\tilde{s})$ is contained in the union of two dyadic intervals $\tilde{s} _L$ and $\tilde{s} _R$ each of size $\lesssim \pi _1 (\tilde{s}) $.  Further, because the set of tiles of a given size and orientation partition $\bbr ^2$ (i.e., for each $\omega \in \cald$, we have 
$\bigcup _{R\in \calu _{\omega}} R = \bbr ^2$ ; see the definitions in Section \ref{modelreduction}), and because $|\pi _1 (s)| \geq |\pi _2 (s)|$ we know that for any dyadic interval $I$, there are $\lesssim 1$ tiles $s\in T$ such that  $I = \pi _1 ( \tilde{ s_L })$ or $I = \pi _1 (\tilde{ s_R })$.  All of this allows us to assume (possibly after dividing $T$ into $\sim 1$ pieces) that the tiles $s$ are parameterized by dyadic intervals, and that for each $x\in C \topp (T)$, and each dyadic interval $I$, there is at most one $s\in T$ such that $x\in \tilde{s}$ and $\pi _1 (\tilde{s}) = I$.

To prove \eqref{outside}, we split the sum inside $\Delta f$ into two pieces, one over tiles whose vertical projection is smaller than the length of $R$, and the other over tiles whose vertical projection is larger than the length of $R$.  We begin by controlling the sum over smaller tiles.  Note that the dominant term in both cases comes from tiles such that $|\pi _1 (s) | \sim | \pi _1 (R)|$.  In the integral below, we need only consider $x\in C\topp (T)$ such that $\pi _1 (x) \not\in \pi _1 (5CR)$.  This allows us to prove the desired estimate using spatial decay alone.  Further, since $\one _{\tilde{s}} (x)$ is constant on vertical segments projecting to 
$\pi _2 (C\topp (T))$, we have 
\beqa
&\left.\right. & \int _{x\in K\topp (T)\cap ( 5CR' )^c } \left(   \sum _{|\pi _1(s)| \leq |\pi _1 (R)|} |\langle b_R , \varphi _s \rangle |^2 {{\one _s } \over {|s|} } \right)^{1\over 2} \\ 
&\lesssim &
\int _{x\in K\topp (T)\cap ( 5CR' )^c } \left(   {{||b_R||_1 ^2 }\over {|R|^2}} \left( {{|x- c(R)| } \over {|\pi _1 (R)|}} \right)^{-10}  \right)^{1\over 2} \\
&\lesssim & {{||b_R|| _1   }\over {|\pi _1 (R)|}} \int _{t\in \bbr \colon |t| \geq 5 |\pi _1 (R)|} {1\over {|{t\over {|\pi _1 (R)|}}|^5} } dt \\
&\lesssim & ||b_R || _1.
\eeqa
We emphasize that the integral in the second-to-last line is one-dimensional.
It remains to control the sum over the tiles with vertical projection larger than $|\pi _1 (R)|$.  This requires using the mean-zero-along-horizontal-line-segments property of the function $b_R$.  Note that for any smooth function $h$, we have
\beqa
\langle b_R, h \rangle &=& \int _{y \in \pi _2 (R)} \int _{x\in \pi _1 (R)} b_R (x,y) h(x,y)  dx dy \\
&\leq & \int _{y \in \pi _2 (R)} \int _{x\in\pi _1 (R)} |b_R (x,y)| |h(x,y) - h(c_{\pi _1 (R)}, y)| dx dy.
\eeqa
Our goal is to apply this to the wave packets $\varphi _s$.  Specifically, we will show
\begin{claim}
\beqa
|\langle b_R, \varphi _s \rangle| & \lesssim & ||b_R||_1 
		{1\over {|s|^{1\over 2} } } {|\pi _1 (R)| \over |\pi _1 (s)| } 
	\min \left(1, \left({{|x - c(R)|} \over |\pi _1 (s)|} \right) ^{-10 } \right)
\eeqa
\end{claim}
\begin{proof}
We must deal with a small technicality here:  the tiles $s$ need not be precisely axis parallel, but fortunately they are close.  Precisely, we have that the vertical component (when using the coordinate frame of $s$) of $(x,y) - (c_{\pi _1(R)}, y)$ is less than $ { w| \pi _1 (R) | \over |\pi _1 (s)| }$ .  Of course we have the horizontal component (when using the coordinate frame of $s$) of $(x,y) - (c_R, y)$ is less than $  | \pi _1 (R) | $.  Further, we know that 
\beqa
D_1 \varphi _s (x,y) & \leq & {1\over \sqrt{|s|}} {1\over |\pi _1 (s)| } {\partial \over \partial x } 
	\varphi ({x\over |\pi _1 (s)|}, {y\over w}) \\
D_2 \varphi _s (x,y) & \leq & {1\over \sqrt{|s|}} {1\over w}  {\partial \over \partial y } 
	\varphi ({x\over |\pi _1 (s)|
}, {y\over w}). \\
\eeqa
Hence 
\beqa
|\varphi _s (x,y) - \varphi _s (c(\pi _1 (R)), y)| & \lesssim &
{ w| \pi _1 (R) | \over |\pi _1 (s)| } {1\over \sqrt{|s|}} {1\over w}  {\partial \over \partial y } 
	\varphi ({x\over |\pi _1 (s)|}, {y\over w}) \\
& + & | \pi _1 (R) | {1\over \sqrt{|s|}} {1\over |\pi _1 (s)| } {\partial \over \partial x } 
	\varphi ({x\over |\pi _1 (s)|}, {y\over w}).
\eeqa
\end{proof}

The claim yields, writing $\Gamma = K\topp (T)\cap ( 5CR' )^c $, 
\beqa
&& \int _{\Gamma } \left(   \sum _{|\pi _1(s)| > |\pi _1 (R)|} |\langle b_R , \varphi _s \rangle |^2 {{\one _{\tilde{s}} (x)} \over {|s|} } \right)^{1\over 2} dx \\
&\lesssim & \int _{ \Gamma }  ||b_R||_1 |\pi _1 (R)|
	\left( \sum _{|\pi _1 (s)| > |\pi _1 (R)|}
	\left( {  \min \left(1, \left({{|x - c(R)|} \over |\pi _1 (s)|} \right) ^{-10 } \right)  \over {|\pi _1 (s)||s|^{1\over 2} } } 
 \right)^2
			{{\one _{\tilde{s}} (x)} \over {|s|} } \right)^{1\over 2} dx \\
&\lesssim &  ||b_R || _1.
\eeqa
This completes the proof of \eqref{outside} and thus the proof of Lemma \ref{lpsquarefcn}.

\section{$BMO$ type estimates for the square function} \label{bmosection}
In this section we prove Lemma \ref{bmosquarefcn}.  As in the previous section, we consider the related operator $\tilde{\Delta}$.  See \eqref{newdelta} for the definition, as well as the discussion immediately following the definition for several simplifying assumptions that we make.  To prove the Lemma, we prove the following key claim.  Here, and in the rest of the proof, we write $\sigma = \size (T)$; note that we also have 
\beqa
\sigma \sim \left( {1\over {|\topp (T)|}} \sum _{s\in T} | \langle f, \varphi _s \rangle | ^2 \right) ^{1\over 2} .
\eeqa
As in the last section, we consider a slightly modified version of $\Delta$:  define 
\beqa
\tilde{\Delta  } f = \left( \sum_{s\in T} |\langle f , \varphi _s \rangle| ^2 {{\one _{\tilde{s} }  } \over {|s|}}  \right)^{1\over 2} 
\eeqa
where the rectangles $\tilde{s}$ are defined immediately above $\eqref{newdelta}$.
\begin{claim} \label{jn}
\beqa
| \{ \tilde{\Delta} f > \sigma n \} | \lesssim 2^{-n^2 } | \{ \tilde{\Delta} f > \sigma  \} |.
\eeqa
(Of course we do not need the full exponential-squared decay, but we do have it.)
\end{claim}
With the Claim, we are almost done:
\beqa
|| \tilde{\Delta} f  || _2 ^2  
&\lesssim & \int_{ \{ \tilde{\Delta} f \leq  \sigma  \} } (\tilde{\Delta} f )^2 
	+ \sum _n \sum _{n=1} ^{\infty} (\sigma n)^2 | \{ \tilde{\Delta} f >  n\sigma  \} |    \\
&\lesssim & \int_{ \{ \tilde{\Delta} f \leq  \sigma  \} } (\tilde{\Delta} f )^2 
	+ \sum _n \sum _{n=1} ^{\infty} (\sigma n)^2 |2^{-n^2} \{ \tilde{\Delta} f >  \sigma  \} |    \\
&\lesssim & \int_{ \{ \tilde{\Delta} f \leq  \sigma  \} } (\tilde{\Delta} f )^2 
	+ \sigma ^2 | \{ \tilde{\Delta} f >  \sigma  \} |    \\
&\lesssim & \sigma \int _{ \{ \tilde{\Delta} f \leq  \sigma  \} }  \tilde{\Delta} f
	+ \sigma \int _{ \{ \tilde{\Delta} f > \sigma  \} }  \tilde{\Delta} f \\
&=&  \sigma \int  \tilde{\Delta} f.
\eeqa
With this, we see that 
\beqa
\sigma ^2 |\topp (T)| \sim || \tilde{\Delta} f  || _2 ^2  \lesssim \sigma \int  \tilde{\Delta} f,
\eeqa
which proves that 
\beqa
|| \tilde{\Delta} f  || _2 \sim \sigma |\topp (T)|^{1\over 2} \lesssim {1\over |\topp (T)|^{1\over 2} } \int \tilde{\Delta} f,
\eeqa
which is what we need.
It remains to prove the claim.

\begin{proof} [Proof of Claim \ref{jn}  ]

Of course to prove the claim it is enough to show that
\beqa
|\{ \tilde{\Delta } f > \sqrt{n} \sigma \} | \lesssim 2^{-n} |\{ \tilde{\Delta } f > \sigma \} |,
\eeqa
and this is equivalent to showing
\beqa
|\{ (\tilde{\Delta } f) ^2 > n \sigma ^2 \} | \lesssim 2^{-n} |\{ (\tilde{\Delta } f) ^2 >  \sigma ^2 \} |  ,
\eeqa
which can be shown in a rather straightforward manner following the proof of the John-Nirenberg inequality.
Recall that for each dyadic $I$ we have an associated tile in $T$, which we call $s(I)$.
For notational convenience, define for intervals $I,K$
\beqa
a_{I, K} (x)= \sum _{I \subseteq J \subseteq K}   |\langle f, \varphi _{s(J)} \rangle |^2  {{\one_s (x)} \over {|s(J)|}} .
\eeqa
We first note that for any $K$, if $I$ is a maximal interval on which
\beqa
a_{I,K}  > m \sigma ^2,
\eeqa
then we know
\beqa
a_{I,K}  < (m+2) \sigma ^2,
\eeqa
since
\beqa
|\langle f, \varphi _{s(I)} \rangle |^2  {1 \over {|s(I)|}} \leq \sigma ^2 .
\eeqa
We begin by defining a collection of intervals $\cali _0$:
\beqa
\cali _0 = \{ \text{ maximal dyadic } I \colon a_{I,\pi _1 (C\topp (T) ) }  > 100\sigma ^2 \} .
\eeqa
Then having defined $\cali _{n-1}$, define for any $K \in \cali _{n-1}$
\beqa
\cali _n (K) &=& \{  \text{ maximal dyadic } I \colon a_{I,K }  > 100\sigma ^2 \} \\
\cali _n &=& \bigcup _{K \in \cali _{n-1} } \cali _n (K) .
\eeqa
We remark that for any $K$,
\beqa
\bigcup _{I \in \cali _n (K) } |I| \leq {1\over {2}} |K|.
\eeqa
To see this we only need to use Chebyshev and the estimate on $\size (T)$:
\beqa
\left| \bigcup _{I \in \cali _n (K) } I \right|
& \leq & {1\over {10\sigma ^2 }} \int  a_{I, K} \\
& \leq & {1\over {10\sigma ^2 }} \sum _{J \subseteq K} |\langle f, \varphi _{s(J)} \rangle |^2 \\
& \leq & {1\over {10}} |K|,
\eeqa
where the last inequality is due to the estimate on $\size (T)$.  Similarly, 
\beqa
|\bigcup _{I\in \cali _0} I | \leq {1\over 2} |\pi _1 (C\topp (T))|.
\eeqa
Putting together all $K$ in $\cali _{n-1}$ gives us that
\beqa
\bigcup _{I \in \cali _n} |I| \leq {1\over 2} \bigcup _{I \in \cali _{n-1} } |I|,
\eeqa
and iterating this gives us that
\beqa
\bigcup _{I \in \cali _n} |I| \leq 2^{-n} \bigcup _{I \in \cali _0} |I|,
\eeqa
which proves Claim \ref{jn} since
\beqa
(\tilde{\Delta }f ) ^2 (x) \lesssim n \sigma ^2
\eeqa
for $x$ such that $\pi _1 (x) \not \in \bigcup _{I \in \cali _n} I $.

\end{proof}


\section{Appendix: The case $p>2$} \label{pgreater2}


In this appendix we briefly discuss the proof of Theorem \ref{main} for $p>2$, 
which is essentially the proof in \cite{LL1}.

Following the tree decomposition of Section \ref{organizationsection} and the remarks
in Section \ref{lemmas}, we need to show   
\beqa
\sum_{\delta }  \sum_{\sigma}  \sum_{T \in \calt _{\delta ,\sigma} }
	 \delta \sigma |\topp (T)| \lesssim |F|^{\frac 1p} |E|^{1-\frac 1p} .
\eeqa
This time we care most about $p$ close to $\infty$.
We may assume $|E|\leq |F|$ because if $|E|>|F|$ then
we may apply the previous arguments for the case $p\leq 2$.  We emphasize here that there is no circularity.  Both the argument in this section (in which we assume $|E|\leq |F|$ ) and the argument in the bulk of the paper (in which we assume $|E|\geq |F|$) work when 
$p=2$.  Hence the $p=2$ case of the estimate in \eqref{modelthm} is established for arbitrary $E,F$.  This allows us to assume $|E|\leq |F|$ in this section, where $p\geq 2$, and allows us to assume $|E|\geq |F|$ in the earlier part of the paper, where $p \leq 2$.

By Estimates \ref {size} and \ref{density} it suffices to prove 
\beqan
\label{plarge}
\sum_{\delta }  \sum_{\sigma}  \sum_{T \in \calt _{\delta ,\sigma} }
	 \delta \sigma 
\min ({{ |E|} \over {\delta }}, {{|F|}\over {\sigma ^2}})
\lesssim |F|^{\frac 1p} |E|^{1-\frac 1p} 
\eeqan
for $p\geq 2$.

The following simple estimate will be helpful:
\begin{claim} \label{basic}
For any $\delta$, we have
\beqa
\sum_{\sigma} 	\delta \sigma \min ({{ |E|} \over {\delta }}, {{|F|}\over {\sigma ^2}}) \lesssim \sqrt{\delta |E||F|}.
\eeqa
\end{claim}
\begin{proof}
We need only observe that the two terms in the minimum are equal when $\sigma = \sqrt{\delta {{|F|} \over {|E|}}  }$ and split the sum over $\sigma$ accordingly.
\end{proof}

We split the sum (\ref{plarge}) in $\delta$ into two pieces, with the dividing line being
$\delta = {{|E|} \over{|F|}}$.  For smaller $\delta$, we use Claim \ref{basic} above:
\beqa
\sum _{\delta \leq {{|E|} \over{|F|}} }
\sum_{\sigma} 	\delta \sigma \min ({{ |E|} \over {\delta }}, {{|F|}\over {\sigma ^2}})
&\lesssim &\sum _{\delta \leq {{|E|} \over{|F|}}}
\sqrt{\delta |E||F|}  \\
&\lesssim &  |E| .
\eeqa
For larger $\delta$, we use the estimate $\size \lesssim 1$:
\begin{claim} \label{sizeclaim}
If the function in the definition of $\size (T)$ is called $f$, then 
\beqa
\size (T) \lesssim ||f||_{\infty}.
\eeqa
\end{claim}
Of course we are using $f = \one _F$, which proves that here $\size (T) \lesssim 1$.
\begin{proof}
For $k\geq 1$, define 
\beqa
\Omega _0 &=& \topp (T) \\
\Omega _k &=& 2^k \topp (T) \setminus 2^{k-1} \topp (T) .
\eeqa
We need only note that for any $1$-tree $T$, by Lemma \ref{bessel}, 
\beqa
\left( \sum _{s\in T} |\langle f, \varphi _s \rangle |^2  \right) ^{1\over 2} 
&\leq &\sum _k  \left( \sum _{s\in T} |\langle \one _{\Omega _k} f, \varphi _s \rangle |^2  \right) ^{1\over 2} \\
&\lesssim & \sum _k 2^{-Nk} ||\one _{\Omega _k} f || _2 ^ 2 \\
& \lesssim & ||f||_{\infty} ^2 |\topp (T)|
\eeqa
since $|\Omega _k | \lesssim 2^{2k} |\topp (T)|$.  This proves the claim.
\end{proof}
Hence
\beqa
\sum _{\delta \geq {{|E|} \over{|F|}} } \sum_{\sigma \leq 1} \delta \sigma {{ |E|} \over {\delta }}
\lesssim |E| \log {{|F|} \over{|E|}}  .
\eeqa
Combining these two estimates proves (\ref{plarge}) since $|E|\leq |F|$.


\begin{thebibliography}{77}


\bibitem{B1} \label{B1}  Bateman, Michael.
{\it $L^p$ estimates for maximal averages along one-variable vector fields in ${\mathbf R} ^2$, }
Proc. Amer. Math. Soc. $\mathbf{137}$,  (2009), 955-963.

\bibitem{B2} \label{B2}  Bateman, Michael.
{\it  Maximal averages along a planar vector field depending on one variable, } to appear in Trans. AMS, 
preprint available at http://lanl.arxiv.org/abs/1104.2859

\bibitem{BT} \label{BT}  Bateman, Michael, and Thiele, Christoph.
{\it  $L^p$ estimates for the Hilbert transform along a one variable vector field , }
preprint available at http://www.math.ucla.edu/~thiele/papers/BTmar2011.pdf



\bibitem{LL1}  \label{LL1} Lacey, Michael, and Xiaochun Li.
{\it   Maximal Theorems for the Directional Hilbert Transform on the Plane }
Trans. Amer. Math. Soc.  358  (2006), 4099-4117.

\bibitem{LL2}  \label{LL2} Lacey, Michael, and Xiaochun Li.
{\it   On a Conjecture of EM Stein on the Hilbert Transform on Vector Fields }
Memoirs of the AMS 205  (2010), no. 965.


\bibitem{LT} \label{LT} Lacey, Michael, and Thiele, Christoph.
{\it A proof of boundedness of the Carleson operator }
Mathematical Research Letters vol. 7 no. 4, (2000), pp 361-370.


\end{thebibliography}
\end{document}